\documentclass[12pt]{amsart}
\usepackage{amssymb,amsfonts,amsmath,amsopn,amstext,amscd,color,latexsym,leqno,mathrsfs,url,xy}
\usepackage{extraipa}
\usepackage{tipa}
\usepackage{undertilde}
\usepackage{ushort}
\input xy
\xyoption{all}

\theoremstyle{remark}

\newtheorem{para}{\bf}[section]

\theoremstyle{definition}

\theoremstyle{plain}

\newtheorem{thm}[para]{Theorem}
\newtheorem{lemma}[para]{Lemma}

\newtheorem{cor}[para]{Corollary}
\newtheorem{prop}[para]{Proposition}

\newenvironment{numequation}
{\addtocounter{para}{1}\begin{equation}}{\end{equation}}

\newcommand{\al}{{\alpha}}

\newcommand{\bbC}{{\mathbb C}}

\newcommand{\bbF}{{\mathbb F}}

\newcommand{\bbN}{{\mathbb N}}

\newcommand{\bbQ}{{\mathbb Q}}
\newcommand{\bbR}{{\mathbb R}}

\newcommand{\bbZ}{{\mathbb Z}}

\newcommand{\bB}{{\bf B}}

\newcommand{\bG}{{\bf G}}

\newcommand{\bN}{{\bf N}}

\newcommand{\bT}{{\bf T}}
\newcommand{\bU}{{\bf U}}

\newcommand{\bb}{{\bf b}}
\newcommand{\bc}{{\bf c}}

\newcommand{\fra}{{\mathfrak a}}

\newcommand{\frg}{{\mathfrak g}}
\newcommand{\frh}{{\mathfrak h}}

\newcommand{\frn}{{\mathfrak n}}

\newcommand{\frt}{{\mathfrak t}}

\newcommand{\frx}{{\mathfrak x}}
\newcommand{\fry}{{\mathfrak y}}

\newcommand{\frY}{{\mathfrak Y}}

\newcommand{\cA}{{\mathcal A}}
\newcommand{\cB}{{\mathcal B}}

\newcommand{\cD}{{\mathcal D}}

\newcommand{\cF}{{\mathcal F}}

\newcommand{\cL}{{\mathcal L}}
\newcommand{\cM}{{\mathcal M}}
\newcommand{\cN}{{\mathcal N}}
\newcommand{\cO}{{\mathcal O}}

\newcommand{\cR}{{\mathcal R}}
\newcommand{\cS}{{\mathcal S}}
\newcommand{\cT}{{\mathcal T}}
\newcommand{\cU}{{\mathcal U}}

\newcommand{\sC}{{\mathscr C}}

\newcommand{\sR}{{\mathscr R}}

\newcommand{\tG}{\tilde{G}}

\newcommand{\Aut}{{\rm Aut}}

\newcommand{\Hom}{{\rm Hom}}

\newcommand{\lra}{\longrightarrow}

\newcommand{\ra}{\rightarrow}

\newcommand{\Gpm}{G^{p^m}}
\newcommand{\Gzpm}{G_0^{p^m}}
\newcommand{\rn}{r_0}
\newcommand{\gr}{gr^\bullet_r\,}
\newcommand{\R}{o_L}
\newcommand{\car}{\stackrel{\cong}{\longrightarrow}}

\DeclareMathOperator*{\limi}{\lim}
\newcommand{\hUgm}{\widehat{U(\frg)_n}}
\newcommand{\hUtm}{\widehat{U(\frt)_n}}
\newcommand{\Ugm}{U(\frg)_n}
\newcommand{\Utm}{U(\frt)_n}
\newcommand{\Uinvm}{(U(\frg)^\bG)_n}

\newcommand{\Ulm}{\cU^{\lambda}_n}\newcommand{\Um}{\cU_n}
\newcommand{\hUlm}{\widehat{\Ulm}}

\newcommand{\hUlmK}{\widehat{\Ulm}_{,K}}\newcommand{\hUmK}{\widehat{\Um}_{,K}}

\newcommand{\tiX}{\tilde{X}}

\newcommand{\tiD}{\tilde{\cD}}
\newcommand{\tiDm}{\tilde{\cD}_n}
\newcommand{\hDml}{\widehat{\cD_n^\lambda}}\newcommand{\Dml}{\cD_n^\lambda}

\newcommand{\hDmlK}{\widehat{\cD_n^\lambda}_{,K}}

\newcommand{\Bslm}{\cA^\lambda_{m,K}}\newcommand{\Blm}{A^\lambda_{m,K}}\newcommand{\Bslmp}{\cA^\lambda_{m+1,K}}\newcommand{\Blmp}{A^\lambda_{m+1,K}}
\newcommand{\Bsl}{\cA^\lambda_{K}}

\newcommand{\teta}{\theta}

\begin{document}

\title{On locally analytic
Beilinson-Bernstein localization and the canonical dimension}
\author{Tobias Schmidt}
\address{Mathematisches Institut, Westf\"alische Wilhelms-Universit\"at M\"unster, Einsteinstr.
62, D-48149 M\"unster, Germany}
\email{toschmid@math.uni-muenster.de}
\date{}
\maketitle

\scriptsize
%\begin{center} PRELIMINARY VERSION \end{center}
\normalsize

\normalsize

\begin{abstract}
Let $\bG$ be a connected split reductive group over a $p$-adic
field. In the first part of the paper we prove, under certain
assumptions on $\bG$ and the prime $p$, a localization theorem of
Beilinson-Bernstein type for admissible locally analytic
representations of principal congruence subgroups %$G$
in the rational points of $\bG$.
%$\bG$
In doing so we take up and extend some recent methods and results
of Ardakov-Wadsley on completed universal enveloping
algebras (\cite{AW}) to a locally analytic setting. %Let $\frg$ be
%the Lie algebra of $\bG$ and $\frt\subset\frg$ a Cartan
%subalgebra. Let $K$ be a $p$-adic coefficient field. Let $\theta$
%be a central character of the universal enveloping algebra
%$U(\frg_K)$ of $\frg_K=\frg\otimes K$ and let $\lambda\in\frt^*_K$
%be associated to $\theta$ via the Harish-Chandra map. We construct
%a sheaf $\cA_K^\lambda$ of (noncommutative) $K$-algebras on the
%flag scheme of $\bG$ reminiscent of classical $\lambda$-twisted
%differential operators. Its global sections are naturally
%isomorphic to $D(G)_\theta$, the central reduction of the algebra
%$D(G)$ of locally analytic distributions on $G$. The sheaf
%$\cA_K^\lambda$ comes with a distinguished full abelian
%subcategory $\sC$ of $\cA_K^\lambda$-modules containing all
%coherent modules. We show, in case $\lambda+\rho$ is dominant and
%regular, that the global section functor induces an equivalence of
%categories between $\sC$ and the category of admissible locally
%analytic $G$-representations over $K$ having infinitesimal
%character $\theta$. In proving the equivalence we build on an
%analogous statement for completed universal enveloping algebras
%proved in \cite{AW}.
As an application we prove, in the second part of the paper, a
locally analytic version of Smith's theorem on the canonical
dimension. %for complex universal enveloping algebras in the context
%of locally analytic representations. %Let $G$ be an arbitrary
%$p$-adic Lie group whose Lie algebra is split semisimple. Let $r$
%be half the dimension of a minimal co-adjoint orbit. According to
%work of A. Joseph the values of $r$ are well-known. We show that
%the canonical dimension of any admissible locally analytic
%$G$-representation which is not zero-dimensional is at least $r$.
%We prove this statement by reducing it to an analogous statement
%for completed universal enveloping algebras proved in \cite{AW}.

%In proving our results we make use of the results of K. Ardakov
%and S. Wadsley (\cite{AW}) on localization for $p$-adically
%completed universal enveloping algebras.
~

\end{abstract}

\tableofcontents

\vskip30pt

\section{Introduction}

%The purpose of this note is to prove an analogue of a classical
%theorem of S.P. Smith (\cite{Smith}) on universal enveloping
%algebras of complex semisimple Lie algebras in the context of
%locally analytic representations of $p$-adic Lie groups.

%\vskip8pt

Let $L/\bbQ_p$ be a finite field extension with ramification index
$e$. Let $\bG$ be a connected split reductive algebraic group over
$L$ with group of rational points $\bG(L)$. In a series of papers
(\cite{ST01a},\cite{ST4},\cite{ST5},\cite{ST6})
Schneider-Teitelbaum have developed a theory of admissible locally
analytic representations of the $p$-adic Lie group $\bG(L)$ in
$p$-adic locally convex vector spaces. Originally constructed for
the reductive group side in the emerging $p$-adic local Langlands
programme (\cite{SchneiderICM}) this theory has found diverse
other arithmetic applications, for example in the field of
$p$-adic automorphic forms and $p$-adic interpolation
(\cite{EmertonO},\cite{Loeffler}). Motivated by the theory of
Harish-Chandra modules for real reductive Lie groups one may ask
whether or not a good localization theory of Beilinson-Bernstein
type exists for such representations. We recall that the classical
Beilinson-Bernstein localization theorem (\cite{BB81}) asserts an
equivalence of categories between representations of a complex
reductive Lie algebra with fixed central character and a category
of twisted $D$-modules on the flag variety. Among other
applications - notably the proof of the Kazhdan-Lusztig
multiplicity conjecture - this theorem can be used to obtain a
classification of the irreducible admissible smooth
representations of a given real reductive Lie group.

\vskip8pt

%More precisely, In \cite{AW} K. Ardakov and S.Wadsley develop
%(among other things) a localization theory for representations of
%$p$-adically completed universal enveloping algebras. In the first
%part of our paper we take up and extend their methods and
%establish a localization for admissible locally analytic
%representations of the principal congruence subgroups of $\bG$.

In the first part of this paper we make a modest step in this
direction and prove a localization theorem for representations of
a prominent series of compact open subgroups of $\bG(L)$, namely
the principal congruence subgroups. In doing so we take up and
extend some recent methods and results of Ardakov-Wadsley on
completed universal enveloping algebras (\cite{AW}) to a locally
analytic setting. To give more details, we fix an extension of
$\bG$ to a split reductive group scheme over $o_L$ and denote it
by the same letter. Our results are valid under three hypotheses
on the geometric closed fibre $\bG_{\bar{s}}$ of $\bG$ which are
familiar from the theory of modular Lie algebras (cf.
\cite{JantzenCharp}): the derived group of $\bG_{\bar{s}}$ is
(semisimple) simply connected, the prime $p$ is good for the
modular Lie algebra $Lie(\bG_{\bar{s}})$ and there exists a
$\bG_{\bar{s}}$-invariant non-degenerate bilinear form on
$Lie(\bG_{\bar{s}})$. A prominent example satisfying these
conditions for all primes $p$ is the general linear group.
Moreover, any almost simple and simply connected $\bG_{\bar{s}}$
satisfies these conditions whenever $p\geq 7$ (and assuming
additionally that $p$ does not divide $n+1$ in case
$\bG_{\bar{s}}$ is of type $A_n$). Assuming these hypotheses in
the following, let $\frg$ be the $o_L$-Lie algebra of the group
scheme $\bG$. For any $k\geq 1$ we have the $k$-th principal
congruence subgroup
\[ G:=\ker\; ( \bG(o_L)\longrightarrow
\bG(o_L/\pi_L^{k}o_L) ).\] We fix a $p$-adic coefficient field
$K$, a finite extension of $\bbQ_p$. Let $U(\frg_K)$ be the
universal enveloping algebra of $\frg_K=\frg\otimes_{\bbZ_p} K$
and let $\theta$ be a central character of $U(\frg_K)$. Let
$Rep(G)_\theta$ be the abelian category of admissible locally
analytic $G$-representations over $K$ having infinitesimal
character $\theta$. On the other hand, let $\bT$ be a maximal
split torus of $\bG$ and $\bB\subset\bG$ a Borel subgroup scheme
containing it. In accordance with the classical situation our
localizations will live on the flag scheme
$$X:=\bG/\bB$$ of $\bG$. Let
$\frt\subset\frg$ be the Lie algebra of $\bT$. Up to a finite
extension of $K$ we may pick a weight $\lambda\in\frt^*_K$ that
maps to $\theta$ under the classical (untwisted) Harish-Chandra
mapping. Let $\rho$ be half the sum over the positive roots.
Suppose $\lambda+\rho$ is dominant and regular. In this situations
our main result is an equivalence of categories
$$\sC_{\cA_K^\lambda}\car Rep(G)_\theta$$
where $\cA^\lambda_K$ is a certain sheaf of noncommutative
$K$-algebras on $X$ and $\sC_{\cA_K^\lambda}$ is an explicitly
given full abelian subcategory of all (left)
$\cA_K^\lambda$-modules. The latter contains, for example, all
coherent modules. The functor is given essentially by the global
section functor. A quasi-inverse can be made explicit using the
central reduction
$$D(G)_\theta:=D(G)\otimes_{Z(\frg_K),\theta} K$$ of
the locally analytic distribution algebra $D(G)$ of $G$. We remark
here that the category $Rep(G)_\theta$ is in natural duality with
a full subcategory of $D(G)_\theta$-modules, the so-called {\it
coadmissible} modules (\cite{ST5}). The sheaf $\cA^\lambda_K$ is a
natural ring extension of a suitable Fr\'echet completion, so to
speak, of the classical $\lambda$-twisted differential operators
$\cD^\lambda_K$ on the generic fibre $X_L$. In fact, there will be
a canonical morphism $j_*\cD^\lambda_K\rightarrow\cA^\lambda_K$
inducing a commutative diagram

\[\xymatrix{
U(\frg_K)_\theta\ar[d]^{\subset}\ar[r]^<<<<{\simeq} &  \Gamma(X_L,\cD^\lambda_K)\ar[d]\\
D(G)_\theta \ar[r]^<<<<{\simeq}  & \Gamma(X,\cA^\lambda_K) }
\]
where $j: X_L\hookrightarrow X$ equals the inclusion of the
generic fibre into $X$ and the upper isomorphism comes from
\cite{BB81}.

\vskip8pt

As in the case of real Lie groups, the above localization theorem
can be used to obtain significant information on the irreducible
objects in the category $Rep(G)_\theta$. We hope to come back to
this in the future.

%The functors yielding the above equivalence are explicitly given
%below. For example, a quasi-inverse is given asMoreover,
%$Z(\frg_K)$ denotes the center of $U(\frg_K)$ and
%$$D(G)_\theta:=D(G,K)\otimes_{Z(\frg_K)} K_\theta$$ denotes the
%central reduction of the algebra of $K$-valued locally analytic
%distributions on $G$. The sheaf $\cA^\lambda_K$ comes equipped
%with a full abelian subcategory of left $\cA^\lambda_K$-modules
%$$\sC_{\cA_K^\lambda}\subset {\rm Mod}(\cA^\lambda_K)$$ containing all coherent
%$\cA^\lambda_K$-modules. Let $Rep(G)_\theta$ be the category of
%admissible locally analytic $G$-representations with infinitesimal
%character $\theta$. Given $\cM\in \sC_{\cA_K^\lambda}$ its space
%of global sections has a natural locally convex topology with
%strong dual $\Gamma(X,\cM)'_b$ a locally analytic
%$G$-representation with character $\theta$. It is our main result
%that this establishes an equivalence of categories
%$$ \sC_{\cA_K^\lambda} \car Rep(G)_\theta.$$ A quasi-inverse can
%be made explicit using the Schneider-Teitelbaum equivalence
%between $Rep(G)_\theta$ and a full abelian subcategory of left
%$D(G)_\theta$-modules (\cite{ST5}). For example, the latter
%contains all finitely presented modules and on such modules a
%quasi-inverse has the form

%$$M\mapsto \cA^\lambda_K\otimes_{D(G)_\theta}M $$ where we regard
%$\cA^\lambda_K$ as a $D(G)_\theta$-module through the above
%commutative diagram.

\vskip8pt

To sketch the construction of the sheaf $\cA^\lambda_K$ we let
$\bN$ be the unipotent radical of $\bB$ and put
$\tilde{X}:=\bG/\bN$. Since $\bT$ normalizes $\bN$ the canonical
projection
$$\xi:\tilde{X}\lra X$$ is a right $\bT$-torsor for the Zariski topology on $X$. There are sheaves $\cD_{\tilde{X}}$
and $\cD_X$ of crystalline (i.e. no divided powers) differential
operators on $\tilde{X}$ and $X$ respectively, familiar from the
theory of arithmetic $D$-modules
(\cite{BerthelotDI},\cite{BerthelotOgus}). We denote by
$$\tilde{\cD}:=\xi_*(\cD_{\tilde{X}})^\bT$$ the relative universal
enveloping algebra of the torsor $\xi$ in the sense of
Borho-Brylinski (\cite{BBII}). Recall that if $U\subset X$ is an
open affine subset trivializing $\xi$ then a choice of such a
trivialization induces an isomorphism
$$\tilde{\cD}(U)\car \cD_X(U) \otimes_{\bbZ_p} U(\frt)$$
for the local sections of $\tilde{\cD}$ above $U$. The
homomorphisms $U(\frt)\ra\tilde{\cD}(U)$ for varying $U$ glue to a
central embedding $U(\frt)\lra\tilde{\cD}$ so that toral weights
give rise to central reductions of $\tilde{\cD}$.

 \vskip8pt

Let $\pi_K$ be a uniformizer of $K$ and $o_K\subset K$ the ring of
integers. In the following we restrict to numbers $m>>0$ such that
$\lambda(\pi_L^m\frt)\subseteq o_K$. For any such number $m$ we
put $n:=(m-1)e+k$ and let $\widehat{\cD_n^\lambda}_{,K}$ be the
central reduction along $\lambda$ of the $p$-adic completion (with
$\pi_K$-inverted) of the $n$-th deformation of the sheaf
$\tilde{\cD}$ as introduced by Ardakov-Wadsley (\cite{AW}). Over
an open affine subset $U\subset X$ trivializing $\xi$ its algebra
of sections is noncanonically isomorphic to a certain $K$-Banach
algebra completion of $\tilde{\cD}_X(U)\otimes_{o_L} K.$ The
Banach norm in question depends on the 'deformation parameter'
$n$. We remark in passing that, in loc.cit., the sheaves
$\widehat{\cD_n^\lambda}_{,K}$ are used to establish a
Beilinson-Bernstein theorem for $p$-adic completions of the
universal enveloping algebra $U(\frg)$. We make heavy use of this
result. In this light our methods and results are simply a direct
and straightforward extension of the corresponding ones in
loc.cit.

\vskip8pt

To go further, the natural $\bG$-equivariant structure of $\cD_X$
extends to $\widehat{\cD_n^\lambda}_{,K}$. Since the latter sheaf
is supported only on the special fibre of $X$ and the latter is
set-theoretically fixed by $G$ we have a group homomorphism
$\sigma: G\ra {\rm Aut}(\widehat{\cD_n^\lambda}_{,K}).$

\vskip8pt

%The group $G$ is $p$-valued in the sense of M. Lazard
%(\cite{Lazard65})
To simplify the exposition in this introduction we assume from now
on $k\geq e$. We then denote by $H_m$ the finite group equal to
the quotient of $G$
modulo its normal subgroup generated by $p^m$-th powers. %Using an
%isomorphism of Lazard %the group $\bG^(p^n)$ can be seen to be
%naturally contained in the global units
%$\Gamma(X,\widehat{\cD_n^\lambda}_{,K})^\times$.
A careful choice of a section $H_m\ra G$ combined with results of
M. Lazard on $p$-valued groups (\cite{Lazard65}) produces from
$\sigma$ a {\it homomorphism}
$$ \sigma_m: H_m\lra {\rm Out}(\widehat{\cD_n^\lambda}_{,K})$$
into the group of outer automorphisms of the sheaf
$\widehat{\cD_n^\lambda}_{,K}$ as well as a $2$-cocycle $$\tau_m:
H_m\times H_m\lra \Gamma(X, \widehat{\cD_n^\lambda}_{,K})^\times$$
with respect to $\sigma_m$. Sheafifying the usual construction
from noncommutative ring theory (\cite{Passman}) we obtain a
crossed product sheaf
$$ \cA^\lambda_{m,K}:=\widehat{\cD_n^\lambda}_{,K}*_{\sigma_m,\tau_m}
H_m$$ on $X$ for any $m>>0$. The latter is a sheaf of
noncommutative associative $K$-algebras naturally containing
$\widehat{\cD_n^\lambda}_{,K}$ as a subsheaf. The sheaves
$\cA^\lambda_{m,K}$ form a projective system with well-behaved
transition maps and the limit
$$\cA_K^\lambda:=\varprojlim_m \cA^\lambda_{m,K}$$ is our promised
sheaf. All sheaves $\cA^\lambda_{m,K}$ are coherent and the
abelian category $\sC_{\cA_K^\lambda}$ arises as a suitable
projective limit construction involving the categories of coherent
$\cA^\lambda_{m,K}$-modules for all $m$. In this situation the
nature of the inclusion ${\rm coh}(\cA^\lambda_K)\subset
\sC_{\cA_K^\lambda}$ measures, so to speak, the failure of the
limit sheaf $\cA^\lambda_K$ to be coherent. For more details we
refer to the main body of this article.

\vskip30pt

In the second part of this paper we apply the methods of the first
part to prove an analogue of a classical theorem of S.P. Smith on
complex universal enveloping algebras (\cite{Smith}) in our
locally analytic context. Let $L/\bbQ_p$ be an arbitrary finite
extension and $G$ an arbitrary Lie group over $L$. The theory of
admissible locally analytic $G$-representations provides a unified
framework for studying finite-dimensional algebraic
representations (if $G$ comes from an algebraic group) as well as
the admissible-smooth representations of Langlands theory. Apart
from the latter two classes there are many representations which
are~{\it genuinely locally analytic}. As a first coarse way to
organize this situation and, more specifically, to meisure the
'size' of the vector space underlying a representation (which
usually is of infinite vector space dimension) one may introduce
an Auslander-Gorenstein style dimension function on the dual
category $\sC$ of coadmissible $D(G)$-modules (cf.
\cite{ST5}\footnote{In [loc.cit.] the authors consider the
associated {\it codimension} function.}). It associates a
well-defined number $0\leq {\rm dim} (M)\leq d$ to each nonzero
coadmissible module $M$ where $d$ equals the dimension of the Lie
group $G$. In this way one obtains a filtration by Serre
subcategories

\[\sC=\sC_{d}\supseteq\sC_{d-1}\supseteq
...\supseteq\sC_{1}\supseteq\sC_{0}\] where $M$ lies in $\sC_{i}$
if and only if ${\rm dim}(M)\leq i$. The modules corresponding to
the aforementioned algebraic representations and the
smooth-admissible representations are concentrated in dimension
zero.

\vskip8pt

Here, we investigate this situation under the additional
assumption that the Lie algebra $Lie(G)$ equals the Lie algebra of
a split reductive group $\bG$ satisfying our three hypotheses
above. It turns out that, in this case, there is a large 'gap' in
the above filtration. Namely, let $\bG'_{\bbC}$ be the complex
derived group of $\bG$ and let $r$ be half the smallest possible
dimension of a non-zero co-adjoint $\bG'_{\bbC}$-orbit. The value
of $r$ depends only on the root system of $\bG'_{\bbC}$ and is
well-known in all cases according to work of A. Joseph
(\cite{Joseph},\cite{JosephMINORBIT}). Our main result says that
if $M$ is a coadmissible $D(G)$-module which is not
zero-dimensional, then ${\rm dim}(M)\geq r$. Moreover, we show
that a coadmissible module is zero-dimensional if and only if its
associated coherent sheaf (in the sense of \cite{ST5}) consists of
finite-dimensional $K$-vector spaces.

\vskip8pt

In \cite{AW} Ardakov-Wadsley prove a version of Smith's theorem
for $p$-adically completed universal enveloping algebras and our
version had its origin in the attempt to generalize their result
to locally analytic representations. The results on locally
analytic distribution algebras as obtained in \cite{SchmidtAUS}
and \cite{SchmidtVECT} enable us to make a rather straightforward
reduction to the case treated in \cite{AW}.

\vskip8pt

The author thanks K. Ardakov for kindly answering some questions
concerning the work \cite{AW} and for his comments on an earlier
version of this article.

%\vskip8pt

%We outline the construction\footnote{In this introduction we assume $p>2$ for convenience; all results hold for $p=2$, with very slight modifications.}. Let $\bG$ be a connected reductive linear algebraic group over a finite extension $L$ of $\Qp$, which we assume to be split for simplicity,
%and put $G = \bG(L)$. Let $\sB$ be the semisimple Bruhat-Tits building of $\bG$, and fix an integer $e \ge 0$. Attached to any facet $F \sub \sB$ is a compact open subgroup $U^{(e)}_F \sub G$, defined as in \cite[I.2]{SchSt97}. For large $e$ the group $U^{(e)}_F$ carries a $p$-valuation (in the sense of \cite{Lazard65}), which is then used to define the completed distribution algebra $D_r(U^{(e)}_F,K)$. From now on, we take $e$ to be sufficiently large. Here, $r$ is a real number in $[p^{-\frac{1}{p-1}},1)$, and $K$ is a fixed finite extension of $L$. For a point $z \in \sB$ we put $U^{(e)}_z = U^{(e)}_F$, where $F$ is the unique facet of $\sB$ containing $z$. The $K$-algebra $D_r(U^{(e)}_z,K)$ naturally contains the universal enveloping algebra $U(\frg)_K = U(\frg) \otimes_L K$, were $\frg = \Lie(\bG)$.

\vskip8pt

%If $V$ is finitely generated then the lattice is finitely generated and hence free (\cite{AW}, Lem. 2.8.

%Let $\bG$ be a connected split semisimple linear algebraic group
%scheme over $o_L$, $\bS$ a maximal $L$-split torus,
%$\Phi=\Phi(\bG,\bS,L)$ the root system relative to $\bS$. Let $G =
%\bG(L), S=\bS(L)$ and $U_\alpha$ for each $\alpha\in\Phi$ be the
%groups of $L$-valued points. Let $\frg = \Lie(\bG)$ be the Lie
%algebra of $\bG$ and $U(\frg)$ the universal enveloping algebra.

%Let $X$ be the (semisimple) Bruhat-Tits building of $G$ and let
%$A$ be the standard apartment of $S$. Let $H:=\{g\in G: gx=x {\rm
%~for~all~}x\in A\}$.

\section {Preliminaries on crossed products and sheaves}\label{sect-crossed}
\renewcommand{\cA}{\mathcal{B}}
\renewcommand{\cB}{\mathcal{A}}

All appearing rings in this section are unital.

\begin{para}\label{para-ACT}
Recall (\cite{Passman}) that a (associative) {\it crossed product}
of a ring $R$ by a group $H$ is an associative ring $R*H$ which
contains $R$ as a subring and contains a set of units
$\overline{H}=\{\overline{h}: h\in H\}$, isomorphic as a set to
$H$, such that
\begin{itemize}
    \item[(a)] $R*H$ is a free left $R$-module with
    basis $\overline{H}$;

    \item[(b)] for all $x,y\in H, \overline{x}R=R\overline{x}$ and
    $\overline{x}\cdot\overline{y}R=\overline{xy}R$.
\end{itemize}
Given such a crossed product one obtains maps $\sigma:
H\rightarrow \Aut(R)$ (an {\it action}) and $\tau:H\times
H\rightarrow R^\times$ (a {\it twisting}) by the rules
\[\begin{array}{ccccc}
  \sigma(x)(r)=\overline{x}^{-1}r\overline{x} &  & {\rm and} & &\overline{x}\cdot\overline{y}=\overline{xy}\tau(x,y). \\
\end{array}\]
It follows that $\sigma$ defines a group homomorphism
$H\rightarrow {\rm Out}(R)$ and that $\tau$ is a $2$-cocycle for
the action of $H$ on $R^\times$ via $\sigma$.\footnote{Note that
we implicitly use the convention of loc.cit. for the
multiplication in ${\rm Aut}(R)$ (and therefore also in ${\rm
Out}(R)$): $(\alpha\beta)(r):=\beta (\alpha(r))$ for $r\in R$ and
two automorphisms $\alpha,\beta$.} Conversely, starting with a
ring $R$, a group $H$, a group homomorphism $\sigma: H\rightarrow
{\rm Out}(R)$ and a $2$-cocycle $\tau:H\times H\rightarrow
R^\times$ one can construct an associative ring
$R*_{\sigma,\tau}H$ which is a crossed product of $R$ by $H$
having the prescribed action and twisting (\cite{Passman}, Lemma
1.1).

\vskip8pt

Given a crossed product $S=R*H$ we can and will always assume that
$\overline{1}$ equals the unit element in $S$ (loc.cit., \S1) and
that the inclusion $R\hookrightarrow S$ is unital. Moreover, if
$R$ is left (right) noetherian and $H$ is finite, then $S$ is left
(right) noetherian (loc.cit., Prop. 1.6).

\vskip8pt

Let from now on $K$ be a field of characteristic zero and $H$ a
finite group. Let $S=R*H$ be a crossed product. We assume that
$R,S$ are $K$-algebras and that $R\hookrightarrow S$ is a
$K$-algebra homomorphism. The following lemma is due to K.
Ardakov. I thank him for allowing me to reproduce it here.
\begin{lemma}\label{Ardakov} Let $R\ra A$ be a ring homomorphism which is (left)
flat and which factores through the inclusion $R\hookrightarrow
S$. The resulting ring homomorphism $S\ra A$ is (left) flat.
\end{lemma}
\begin{proof}
Let $M$ be a right $S$-module. It suffices to show ${\rm
Tor}_1^{S}(M,A)=0$. Let $K[H]$ be the usual group algebra. There
is a ring homomorphism $\Delta: S\ra S \otimes_K K[H]$ induced by
$\bar{h}\mapsto \bar{h}\otimes h$. The tensor product $M\otimes_K
K[H]$ is a right $S$-module via $\Delta$. Now $\Delta$ exhibits
$R\hookrightarrow S$ as a (right) $K[H]$-Galois extension (e.g.
\cite{AWCartanMap}, 2.2). According to loc.cit., Prop. 2.3 (a) the
$S$-module $M\otimes_K K[H]$ is therefore isomorphic to
$M\otimes_R S$. On the other hand, the semisimplicity of $K[H]$
together with loc.cit., Lemma 2.4 implies that $M$ is an
$S$-module direct summand of $M\otimes_K K[H].$ If $N$ is a
complementary submodule we therefore obtain
$$ {\rm Tor}_1^{S}(M, A) \oplus {\rm Tor}_1^{S}(N,A) = {\rm Tor}_1^{S}( M
\otimes_K K[H], A) \simeq {\rm Tor}_1^{S}(M\otimes_R S, A) = {\rm
Tor}_1^R(M,A)=0$$ where the last two identities follow from
flatness of $R\ra S$ and $R\ra A$ respectively.
\end{proof}

%\begin{prop}
%noetherianness,..., flatness,..., ideals.
%\end{prop}
\end{para}

\begin{para}
Let $X$ be a topological space and $\cA$ be a sheaf of not
necessarily commutative rings on $X$. Let $H$ be a finite group.
There is an obvious notion of (associative) crossed product of the
sheaf $\cA$ by $H$ by which we mean a sheaf $\cA*H$ of associative
rings on $X$ which contains $\cA$ as a subsheaf and has a
distinguished set of global units
$$\overline{H}=\{\overline{h}: h\in H\}\subseteq
\Gamma(X,\cA*H)^\times,$$ isomorphic as a set to $H$, such that
\begin{itemize}
    \item[(a)] $\cA*H$ is a free left $\cA$-module with
    basis $\overline{H}$ (i.e. the natural map
    $\oplus_{|H|}\cA\ra\cA*H$ given by the global sections
    $\overline{h}$ is an isomorphism of $\cA$-modules),

    \item[(b)] for all $x,y\in H, \overline{x}\cA=\cA\overline{x}$ and
    $\overline{x}\cdot\overline{y}\cA=\overline{xy}\cA$.
\end{itemize}
Given such a crossed product one obtains maps $\sigma:
H\rightarrow \Aut(\cA)$ and $\tau:H\times H\rightarrow
\Gamma(X,\cA)^\times$ as before, i.e. by the rules
$\sigma(x)(s)=\overline{x}^{-1}s\overline{x}$ and
$\overline{x}\cdot\overline{y}=\overline{xy}\tau(x,y)$ for any
local section $s$. It follows that $\sigma$ defines a group
homomorphism $H\rightarrow {\rm Out}(\cA)$ and $\tau$ is a
$2$-cocycle for the action of $H$ on $\cA^\times$ via $\sigma$.
Here, ${\rm Out}(\cA)$ refers to the quotient of ${\rm Aut}(\cA)$
by its normal subgroup of automorphisms arising via conjugation by
a global unit. Furthermore, to match with our convention in the
case of rings, we define here $(\alpha\beta)(s):=\beta(\alpha(s))$
for a local section $s$ of $\cA$ and automorphisms $\alpha,\beta$.

Conversely, starting with a sheaf $\cA$, a group $H$, a group
homomorphism $\sigma: H\rightarrow {\rm Out}(\cA)$ and a
$2$-cocycle $\tau:H\times H\rightarrow \Gamma(X,\cA)^\times$ one
can construct a sheaf of associative rings $\cA*_{\sigma,\tau}H$
which is a crossed product of $\cA$ by $H$ having the prescribed
action and twisting. The construction is a straightforward
sheafification of the usual argument and is left to the reader.

\vskip8pt

Given a crossed product $\cA*H$ we will always assume that
$\overline{1}$ equals the global unit $1\in\Gamma(X,\cA*H)^\times$
and that the inclusion $\cA\hookrightarrow\cA*H$ is unital.
\end{para}
\begin{para}\label{para-COH}
Let again $\cA$ be a sheaf of not necessarily commutative rings on
$X$. Recall (\cite{EGA_I}, 0.5.3.1) that a left $\cA$-module is
{\it coherent} if it is of finite type and if for all open subsets
$U\subset X$ and all morphisms $\varphi: (\cA|_U)^b\ra \cM|_U$ the
kernel of $\varphi$ is again of finite type. The coherent
$\cA$-modules form a full abelian subcategory ${\rm coh}(\cA)$ of
all $\cA$-modules. We say $\cA$ is a {\it sheaf of coherent rings}
if $\cA$ is coherent as a left module over itself. If $\cA$ is a
sheaf of coherent rings, then a $\cA$-module is coherent if and
only if it is of finite presentation (loc.cit., 0.5.3.7).

\end{para}

\begin{para}
We suppose we are now given a projective system of sheaves $\cA_n,
n\geq 0$ of not necessarily
commutative $K$-algebras on $X$. %Let ${\rm coh}(\cA_n)$ be the abelian
%category of coherent $\cA_n$-modules.
We assume the following two local properties hold for the sheaves
$\cA_n$. Any point $x\in X$ has a basis $\cS_x$ of open
neighbourhoods $U\subseteq X$ such that
\begin{itemize}
    \item[(i)] $\Gamma(U,\cA_n)$ is a left
    and right noetherian $K$-Banach algebra,

    \item[(ii)] the transition homomorphism $\Gamma(U,\cA_{n+1})\ra\Gamma(U,\cA_n)$ is continuous and is left and
right flat with dense image.
\end{itemize}
We remark here that a straightforward generalization of
\cite{BGR}, Prop. 1.2.1/2 to the noncommutative setting shows that
the given norm in (i) can always be replaced by an equivalent one
which is submultiplicative.

The union $\cS:=\cup_{x\in X}\cS_x$ is a base for the topology on
$X$.
%The first two properties imply that each $\cA_n$ is coherent
%(\cite{BerthelotDI}, Prop. 3.1.1).
We let
$$\cA:=\varprojlim_n\cA_n$$ be the projective limit of the system
$(\cA_n)_n$. Note that
$\Gamma(U,\cA)=\varprojlim_n\Gamma(U,\cA_n)$ for all open subsets
$U\subseteq X$ (\cite{EGA_I}, 0.3.2.6). In the following we often
abbreviate $\cF(U):=\Gamma(U,\cF)$ for a sheaf $\cF$ on $X$ and an
open subset $U\subseteq X$.

\end{para}
\begin{para}

Let ${\rm Mod}(\cA)$ be the abelian category of (left)
$\cA$-modules on $X$. Our definitions allow the following simple
sheafification of the formalism of coadmissible modules as
developed in \cite{ST5}, \S3. In particular, we will produce a
certain full abelian subcategory
$${\rm coh}(\cA)\subset \sC_{\cA}\subset {\rm Mod}(\cA).$$ To do this consider the category of projective
systems $(\cM_n)_n$ of coherent
 $\cA_n$-modules $\cM_n$ with the
property that the transition maps induce isomorphisms
$$ \cA_n \otimes_{\cA_{n+1}} \cM_{n+1}\car \cM_n$$
of $\cA_n$-modules.
%Note that for all $U\in\cS$ and $n\geq 0$ the (left)
%$\cA_n(U)$-module $\cM_n(U)$ is finitely generated,
With the usual notion of morphism these projective systems form a
category ${\rm coh}((\cA_n)_n)$. As a consequence of the flatness
requirement (ii) the base change functor
\begin{numequation}\label{equ-exactbasechange}\cA_n\otimes_{\cA_{n+1}}(\cdot): {\rm Mod}(\cA_{n+1})\ra{\rm
Mod}(\cA_{n})\end{numequation} is exact. In view of the noetherian
hypothesis in (i) the category ${\rm coh}((\cA_n)_n)$ is therefore
abelian. We have an additive and left exact functor
$$\Gamma(\cM_n):=\varprojlim_n \cM_n$$ into ${\rm Mod}(\cA)$.
Borrowing terminology from loc.cit. a $\cA$-module $\cM$ will be
called {\it coadmissible} if it is isomorphic to a module of the
form $\Gamma(\cM_n)$ for some $(\cM_n)_n\in {\rm coh}((\cA_n)_n)$.
We let $\sC_\cA$ be the full subcategory of ${\rm Mod}(\cA)$
consisting of coadmissible modules.
\begin{prop}\label{prop-co}
The functor $\Gamma$ is exact. For any $(\cM_n)_n$ in ${\rm
coh}((\cA_n)_n)$ and $M=\Gamma(\cM_n)$ the natural map
$$ \cA_n\otimes_{\cA}\cM\car \cM_n$$ is an isomorphism for any
$n\geq 0$. For any $U\in\cS$ and for any $n\geq 0$ the ring
homomorphism $\cA(U)\ra\cA_n(U)$ is left and right flat.
\end{prop}
\begin{proof}
This is a straightfoward adaption of the arguments in Theorem A
and B and Cor. 3.1 of loc.cit. as follows. Given $\cM_n$ and a
point $x\in X$ there is an open neighbourhood $U$ of $x$ such that
any $\cM_n(U)$ is a finitely generated module over the noetherian
Banach algebra $\cA_n(U)$. It therefore has a canonical Banach
topology (loc.cit., Prop. 2.1). By the density requirement in (ii)
the projective system $(\cM_n(U))_n$ has the property that
$\cM_{m}(U)\ra\cM_n(U)$ has dense image for all $m>n$. In
particular, the map $M(U)\ra\cM_n(U)$ has dense image where
$M:=\Gamma(\cM_n)$. It now follows that, given a surjection
$(\cM_n)_n\ra (\cN_n)_n$ in ${\rm coh}((\cA_n)_n)$ the map
$M(U)\rightarrow N(U)$ is surjective and
$$ \cA_n(U)\otimes_{\cA(U)}\cM(U)\car \cM_n(U)$$ is an
isomorphism. Letting $U$ run through a neighbourhood basis for $x$
gives the first two claims. The last assertion follows similarly
(cf. loc.cit., Remark 3.2).
\end{proof}
\begin{cor}
The functor $$\Gamma: {\rm coh}((\cA_n)_n)\car\sC_\cA$$ is an
equivalence of categories. \end{cor}
\begin{proof}
By definition the functor $\Gamma: {\rm coh}((\cA_n)_n)\ra\sC_\cA$
is essentially surjective. According to the preceding proposition
it is fully faithful.
\end{proof}
\begin{prop}
Any $\cA$-module of finite presentation is coadmissible. In
particular, ${\rm coh}(\cA)\subset \sC_\cA$.
\end{prop}
\begin{proof}
Let $V\subseteq X$ be an open set. By the local nature of our
conditions imposed we may apply the above construction to the
projective limit sheaf $\cA|_V=\varprojlim_n (\cA_n|_V)$ and
obtain the abelian category of coadmissible $\cA|_V$-modules on
$V$. Obviously, a given $\cA$-module $\cM$ is coadmissible if and
only if this is true for $\cM|_V$ for all open sets $V\subseteq
X$. Suppose $\cM$ is of finite presentation. Locally, $\cM$ equals
the cokernel of a morphism of the type $\cA^{a}\ra \cA^b$ for
numbers $a,b\geq 0$ and is therefore coadmissible.
\end{proof}
We finally suppose that our chosen base $\cS$ for the topology on
$X$ is such that
\begin{itemize}
    \item[(iii)] for all $U,V\in\cS$ with $V\subset U$
    the restriction homomorphism $$\Gamma(U,\cA_{n})\ra\Gamma(V,\cA_n)$$ is left and right
    flat for all $n\geq 0$.
\end{itemize}
 According to (i) and \cite{BerthelotDI}, Prop. (3.1.1) the sheaf $\cA_n$
 is then
 a sheaf of coherent rings for all $n$. Of course, this does {\it not}
imply the
 coherence of the sheaf $\cA$.
\end{para}
\begin{para}
We keep the notation of the preceding paragraph. For each $n\geq
0$ we now assume additionally that we are given a finite group
$H_n$ and a crossed product
$$\cB_n=\cA_n*_{\sigma_n,\tau_n}H_n$$ with the properties: $\cA_n\hookrightarrow\cB_n$ maps $K$ into the center and
for $x\in X, U\in\cS_x$ and  $h\in H$ the algebra automorphism
$$\sigma_n(h):\cA_n(U)\car\cA_n(U)$$ is an isometry of the Banach algebra
$\cA_n(U)$ with respect to a defining Banach norm on $\cA_n(U)$
which is {\it submultiplicative} (cf. 2.5). We further suppose
that this collection of sheaves $(\cB_n)_n$ of $K$-algebras forms
a projective system and that the transition morphism
$\cB_{n+1}\ra\cB_n$ is compatible with the map $\cA_{n+1}\ra\cA_n$
for all $n$.
\begin{prop}\label{prop-hypothesis}
The system $(\cB_n)_n$ satisfies conditions (i), (ii) and (iii).
\end{prop}
\begin{proof}
Let $x\in X$ and $U\in\cS_x$. We have $\cB_n(U)=\cA_n(U)*H_n$
which is a noetherian ring (cf. 2.1). Let $|.|$ denote the chosen
submultiplicative Banach norm on $\cA_n(U)$. Then $\cB_n(U)$ has a
natural Banach space norm given by the maximum norm $q(\cdot)$
with respect to the $\cA_n(U)$-module basis $\{\overline{h}: h\in
H_n\}$. By assumption we have
$$|\overline{h}^{-1}x\overline{h}|=|\sigma_n(h)(x)|=|x| \hskip30pt {\rm (*)}$$ for $h\in H_n,
x\in\cA_n(U).$ Let $\bar{h}_1,...,\bar{h}_s$ be an enumeration of
the finitely many elements $\{ \bar{h}: h\in H_n\}$. Let us write
$\mu=\sum_k \lambda_k(\mu)\overline{h}_k$ for an arbitrary element
$\mu\in \cB_n(U)$ with coefficients $\lambda_k(\mu)\in\cA_n(U)$.
In particular, $q(\mu)=\max_k |\lambda_k(\mu)|$. Applying this to
the product $\bar{h}_i \bar{h}_j\in \cB_n(U)$ defines coefficients
$\lambda_k (\bar{h}_i\bar{h}_j)$. For two arbitrary elements
$\mu,\mu'\in\cB_n(U)$ we now compute
$$
\begin{array}{ccl}
  \mu\mu'=\sum_{i,j}\lambda_i(\mu)\bar{h}_i\lambda_j(\mu')\bar{h}_j
& = &  \sum_{i,j}\lambda_i(\mu)
\sigma_n(h_i)^{-1}(\lambda_j(\mu'))\bar{h}_i\bar{h}_j \\
   &  &  \\
   & = & \sum_{i,j,k}\lambda_i(\mu)
\sigma_n(h_i)^{-1}(\lambda_j(\mu'))\lambda_k(\bar{h}_i\bar{h}_j)\bar{h}_k. \\
\end{array}
$$

For the value of $q(\mu\mu')$ we therefore find
$$
\begin{array}{ccl}
\max_{i,j,k} |\lambda_i(\mu)
\sigma_n(h_i)^{-1}(\lambda_j(\mu'))\lambda_k(\bar{h}_i\bar{h}_j) |
& \leq &\max_{i,j,k} |\lambda_i(\mu)|\cdot
|\sigma_n(h_i)^{-1}(\lambda_j(\mu'))|\cdot
|\lambda_k(\bar{h}_i\bar{h}_j) |
 \\
   &  &  \\
   & \leq  &
q(\mu)q(\mu')\max_{i,j,k}
|\lambda_k(\overline{h}_i\overline{h}_j)|\\
\end{array}
$$
using ${\rm (*)}$ for the last inequality. This shows that the
ring multiplication on $\cB_n(U)$ is continuous with respect to
the Banach topology coming from $q(\cdot)$. In other words,
$\cB_n(U)$ is a noetherian Banach algebra which yields (i). We
furthermore have a commutative diagram of rings
\[\xymatrix{
\cA_{n+1}(U)\ar[d]\ar[r] &  \cA_n(U)\ar[d]\\
 \cB_{n+1}(U) \ar[r]  & \cB_n(U) }
\]
in which the upper horizontal arrow is flat with dense image and
the vertical arrows are finite free ring extensions. By definition
of the Banach topologies the lower horizontal arrow has dense
image. According to Lemma \ref{Ardakov} it is flat whence (ii).
Applying loc.cit. to the lower horizontal arrow in the commutative
diagram
\[\xymatrix{
\cA_{n}(U)\ar[d]\ar[r]^{res} &  \cA_n(V)\ar[d]\\
 \cB_{n}(U) \ar[r]^{res}  & \cB_n(V) }
\]
for $U,V\in\cS$ with $V\subset U$ finally yields (iii).
\end{proof}
Let $\cB:=\varprojlim_n\cB_n.$ By the above proposition we
therefore have an equivalence between abelian categories
\begin{numequation}\label{equ-equivalence}\Gamma: {\rm
coh}((\cB_n)_n)\car\sC_\cB\end{numequation} and each sheaf $\cB_n$
is coherent. Moreover, ${\rm coh}(\cB)\subset\sC_{\cB}.$
%We remark that for $M\in\sC_\cA$ the base change
%$\cB\oplus_{\cA}M$ usually does {\it not} lie in $\sC_\cB$.
\end{para}

\section{Differential operators on homogeneous spaces}\label{sect-BB}

\renewcommand{\cA}{\mathcal{A}}
\renewcommand{\cB}{\mathcal{B}}

From now on we will use the following notation: $p$ is a prime
number in $\bbZ$ and $\bbQ_p\subseteq L\subseteq K$ denotes a
chain of finite extensions of $\bbQ_p$. The absolute value $|.|$
on $K$ is normalized by $|p|=p^{-1}$. Let $\R\subseteq L$ be the
ring of integers and $\pi_L\in \R$ a uniformizer. Let
$\kappa:=o_L/(\pi_L)$ denote the residue field of $L$. Let
$[L:\bbQ_p]$ and $e$ be the degree and the ramification index of
the extension $L/\bbQ_p$ respectively. Also, $o_K\subset K$
denotes the integers in $K$ and $\pi_K\in o_K$ a uniformizer. An
$o_K$-submodule of a vector space $V$ over $K$ is called a {\it
lattice} if it contains a basis of $V$ over $K$.

\vskip8pt

We recall some notions and constructions related to differential
operators on homogeneous spaces
(\cite{AW},\cite{BMR08},\cite{BBII}) thereby fixing some notation.

\begin{para}
Let $\bG$ be a connected split reductive group scheme over $o_L$.
We denote its group of $o_L$-valued points by $\bG(o_L)$.
%We assume $\bG$ is simply connected.
Let $\bT\subset\bG$ be a
maximal torus with Lie algebra $\frt$. Let $X^*(\bT)$ and
$X_*(\bT)$ be the group of algebraic characters respectively
cocharacters of $\bT$ with the usual perfect pairing
$$\langle \cdot,\cdot\rangle: X^*(\bT)\times X_*(\bT)\lra\bbZ$$ defined by $z^{\langle a,b\rangle}:=a\circ b(z). $
Let $\Phi\subset X^*(\bT)$ denote the root system determined by
the adjoint action of $\bT$ on the Lie algebra $\frg$ of $\bG$.
Let $W$ denote the corresponding Weyl group. For any $w\in W$ we
fix a representative $\dot{w}$ in $\bG(o_L).$
%(\cite{Jantzen}, II.1.4 (3)).

\vskip8pt

We choose a Borel subgroup scheme $\bB\subset\bG$ containing $\bT$
and let $\Phi^+\subset\Phi$ be the associated subset of positive
roots. We write $\bN$ for the unipotent radical of $\bB$. We
identify, once and for all, the torus $\bT$ with the abstract
Cartan subgroup $\bB/\bN$ via the morphism
$\bT\subset\bB\rightarrow\bB/\bN$ where the second map equals the
canonical projection. The group $W$ acts naturally on the spaces
$\frt^*:=\Hom_{o_L}(\frt,o_L),~\frt^*_L:=\frt^*\otimes_{o_L} L$
and $\frt_L:=\frt\otimes_{o_L} L$. Via differentiation $d:
X^*(\bT)\hookrightarrow\frt^*$ we view $X^*(\bT)$ as a subgroup of
$\frt^*$. We have that $X^*(\bT)\otimes_{\bbZ}L=\frt^*_L$ and, via
$\langle .,.\rangle$ and base change from $\bbZ$ to $L$, that
$X_*(\bT)\otimes_{\bbZ}L=\frt_L$. Let
$\rho=\frac{1}{2}\sum_{\alpha\in\Phi^+}\alpha.$ Let
$\check{\alpha}$ be the coroot of $\alpha\in\Phi$ viewed as an
element of $\frt_L$. An arbitrary weight $\mu\in\frt^*_L$ is
called {\it dominant} if
$$\mu(\check{\alpha}) \notin \{-1,-2,-3,...\}$$ for all
$\alpha\in\Phi^+$. %Note that this notion generalizes the usual
%notion of dominance for integral weights in $\frt^*_L$.
The weight $\mu$ is called {\it regular} if its stabilizer under
the $W$-action is trivial.

\end{para}

\begin{para}
We write $\bN$ for the unipotent radical of $\bB$ and denote by
$\bN^{-}$ the unipotent radical of the Borel subgroup scheme
opposite to $\bB$. The schemes $\bB$ and $\bN$ act via right
translations on $\bG$ and we put
\[
\begin{array}{cccc}
  \tiX:=\bG/\bN, &  &  & X:=\bG/\bB$$  \\
\end{array}
\]

for the corresponding quotients. These are smooth and separated
schemes over $o_L$. The canonical projection
$\xi:\tiX\longrightarrow X$ is a smooth morphism. Since $\bT$
normalizes $\bN$ the scheme $\tiX$ has a right $\bT$-action making
$\xi$ a $\bT$-torsor for the Zariski topology on $X$ (in the sense
of \cite{Milne}, III.\S4). Indeed, a covering of $X$ which
trivializes $\xi$ is given by the open subschemes $\bU_w, w\in W$
where
$$\bU_w:= {\rm image~of~}\dot{w}\bN^{-}\bB$$
under the canonical projection $\bG\rightarrow \bG/\bB$
(\cite{Jantzen}, II.(1.10)). We let $\cS$ denote the collection of
affine open subsets of $X$ over which $\xi$ becomes trivial. By
what we have just said it is a base for the Zariski topology on
$X$.

\end{para}

\begin{para}
Let $\cO_{\tiX}$ be the structure sheaf of $\tiX$ and let
$\cT_{\tiX}={\mathcal Der}_{o_L}(\cO_{\tiX})$ be the tangent sheaf
of $\tiX$ (\cite{EGA_IV_4}, 16.5.7). Using the natural action of
$\cT_{\tiX}$ on $\cO_{\tiX}$ by derivations we have the
semi-direct product $\cO_{\tiX}\oplus \cT_{\tiX} $ of $\cT_{\tiX}$
with the abelian Lie algebra $\cO_{\tiX}$, a sheaf of Lie algebras
over $o_L$. The corresponding universal enveloping algebra is
called the sheaf of {\it crystalline differential operators (of
level zero)} on $\tiX$ (\cite{BerthelotOgus}, \cite{BMR08}).
Following \cite{AW}, 4.2 we denote it by $\cD_{\tiX}.$ It is a
sheaf of $o_L$-algebras and, at the same time, an
$\cO_{\tiX}$-module through the map
$$\cO_{\tiX}\lra \cO_{\tiX}\oplus \cT_{\tiX}, f\mapsto (f,0).$$

There is a positive increasing $\bbZ$-filtration on $\cD_{\tiX}$
$$ 0\subset F_0\cD_{\tiX}\subset F_1\cD_{\tiX}\subset
F_2\cD_{\tiX}\subset\cdot\cdot\cdot$$ consisting of coherent
$\cO_{\tiX}$-submodules, such that

\[\begin{array}{ccccc}

F_0\cD_{\tiX}=\cO_{\tiX}, &
F_1\cD_{\tiX}=\cO_{\tiX}\oplus\cT_{\tiX} & {\rm and } &
F_m\cD_{\tiX}= F_1\cD_{\tiX}\cdot F_{m-1}\cD_{\tiX} &
{\rm~for~}m>1.\\
\end{array}\]
Given two local sections $\partial$ and $f$ of $\cT_{\tiX}$ and
$\cO_{\tiX}$ respectively we have
$$\partial\cdot f-f\cdot\partial=\partial(f)$$
for their commutator in $\cD_{\tiX}$. This means that the
associated graded sheaf of $\cD_{\tiX}$ is canonically isomorphic
to the symmetric algebra of the locally free $\cO_{\tiX}$-module
$\cT_{\tiX}$,
\begin{numequation}\label{equ-gradedunivenvelope}
gr_\bullet\cD_{\tiX}\car {\rm
Sym}_{\cO_{\tiX}}\cT_{\tiX}.\end{numequation} We have an obvious
morphism of $\cD_{\tiX}$ to the sheaf of usual differential
operators on $\tilde{X}$ (\cite{EGA_IV_4}, 16.8). Since
$char(L)=0$ it is an isomorphism over the generic fibre of
$\tilde{X}$.

\vskip8pt

 By the same token there are
corresponding sheaves of crystalline differential operators
$\cD_X$ and $\cD_{\bT}$ on the smooth $o_L$-schemes $X$ and $\bT$
respectively.

\vskip8pt

%Remark: Choosing local coordinates $t_1,...,t_m$ on ${\tiX}$ with
%corresponding derivations $\partial_1,...,\partial_m$ any
%crystalline differential operator has a representation as a sum
%$\sum_{\alpha\in\bbN_0^m}f_\alpha
%\partial^{\alpha}$ with unique local sections $f_\alpha$ of $\cO_{\tiX}$ and
%where
%$\partial^{\alpha}=\partial_1^{\alpha_1}\cdot\cdot\cdot\partial_m^{\alpha_m}.$
%The filtration is given in these coordinates by 'order of
%differential operators'
%$${\rm deg}(\sum_{\alpha\in\bbN_0^m}f_\alpha
%\partial^{\alpha})=\max\{|\alpha|: \partial^\alpha\neq 0\}$$
%where $|\alpha|=\alpha_1+\cdot\cdot\cdot+\alpha_m$.

\end{para}\begin{para}

We let
$$\tilde{\cD}:=(\xi_*(\cD_{\tiX}))^\bT$$ denote the {\it Borho-Brylinski relative
enveloping algebra} of the $\bT$-torsor $\xi$ (\cite{BBII}, \S1).
Again, this is a sheaf of $o_L$-algebras and, at the same time, an
$\cO_{X}$-module through the map
$$\xi^{\sharp}:\cO_X\lra(\xi_*\cO_{\tiX})^\bT.$$ It has a positive increasing $\bbZ$-filtration
$$F_m\tilde{\cD}:=(\xi_*(F_m\cD_{\tiX}))^\bT$$
induced by the filtration $F_\bullet\cD_{\tiX}$. Given an open
subset $U$ from $\cS$ a choice of trivialization
$$\xi^{-1}(U)\simeq U\times_{o_L}\bT$$ over $U$ induces an
isomorphism of $o_L$-algebras
$$\tilde{\cD}(U)\car\cD (U\times_{o_L}\bT)^\bT=\cD_X(U)\otimes_{o_L} U(\frt)$$
where the right-hand side is the usual tensor product of
$o_L$-algebras and where we use the symbol $\cD$ to denote the
crystalline differential operators on the smooth $o_L$-scheme
$U\times_{o_L}\bT$. The last identity follows here from
$$\cD (U\times_{o_L}\bT)^\bT=(\cD_X(U)\otimes_{o_L}\cD_{\bT}(\bT))^\bT=\cD_X(U)\otimes_{o_L}\cD_{\bT}(\bT)^\bT$$
using (\cite{Jantzen}, I.2.10(3)) together with the well-known
isomorphism $U(\frt)\car\cD_{\bT}(\bT)^\bT$ for the split torus
$\bT$ (loc.cit., I.7.13).

\end{para}\begin{para}

The group $\bG$ acts on $\tiX$ and $X$ by left translations and
$\xi$ is $\bG$-equivariant. The sheaf $\cD_{\tiX}$ has a natural
$\bG$-equivariant structure inherited from the usual equivariant
structures of $\cO_{\tiX}$ and $\cT_{\tiX}$. Since the right
$\bT$-action on $\tiX$ commutes with the left $\bG$-action the
sheaf $\tiD$ is naturally $\bG$-equivariant.

\vskip8pt

In the following we denote the constant sheaf with fiber $U(\frg)$
on $\tiX$ again by $U(\frg)$ (and similarly in other cases).
Differentiating the left $\bG$-action on $\tiX$ gives a
homomorphism $U(\frg)\ra \cD_{\tiX}$ which is $\bG$-equivariant
with respect to the adjoint action of $\bG$ on $U(\frg)$
(\cite{DemazureGabriel}, II.\S4.4.4). Whenever the open set
$V\subset\tiX$ is right $\bT$-stable the section of this morphism
over $V$ has image in $\cD_{\tiX}(V)^\bT$. Letting $V$ run through
the subsets $\xi^{-1}(U)$ for $U\in\cS$ one obtains a morphism
\begin{numequation}\label{equ-varphi}\varphi: U(\frg)\lra\tiD\end{numequation} which is $\bG$-equivariant
with respect to the adjoint action of $\bG$ on $U(\frg)$.
Similarly, differentiating the right $\bT$-action on $\tiX$
induces a homomorphism \begin{numequation}\label{equ-psi}\psi:
U(\frt)\lra\tiD\end{numequation} which is a central embedding. Its
image lies in the $\bG$-invariants of $\tiD$ (\cite{Milicic93},
\S3).

\end{para}

\section{Deformations and completions}\label{sect-DC}

We keep the notation from the preceding section and recall a
Beilinson-Bernstein type equivalence of categories for
$p$-adically completed universal enveloping algebras as
constructed in \cite{AW}.
%To ease notation we let $\pi:=\pi_L$ in the following.
\begin{para}
Let as above $\frg$ and $\frt$ be the $o_L$-Lie algebras of the
group schemes $\bG$ and $\bT$ respectively. Fix a number $n\geq
0$. We denote by $U(\cdot)$ the universal enveloping algebra of
whatsoever Lie algebra we wish to consider. Choosing an
$o_L$-basis for $\frg$ the algebra $U(\frg)$ is endowed with a
positive filtration $F_\bullet U(\frg)$, its usual PBW-filtration
(\cite{Dixmier}, 2.3). The associated graded algebra equals the
symmetric algebra $S(\frg)$ on $\frg$. According to these
properties $U(\frg)$ is therefore a {\it deformable} $o_L$-algebra
in the sense of \cite{AW}, Def. 3.5. Its $n$-{\it th deformation}
(loc.cit.) is the $o_L$-submodule
$$U(\frg)_n:=\sum_{i\geq 0} \pi_L^{in}F_iU(\frg)$$
of $U(\frg)$. It is easily seen to be equal to $U(\pi_L^{n}\frg)$
and is therefore even a subalgebra of $U(\frg)$. We denote by
$$\hUgm=\varprojlim_m U(\frg)_n/p^mU(\frg)_n$$ its $p$-adic completion.
In the same way we have the algebra $U(\frt)_n$ and its $p$-adic
completion $\hUtm$.
\end{para}

\begin{para}
As already recalled above the group $\bG$ acts on $U(\frg)$ by the
usual adjoint representation. The ring of invariants $U(\frg)^\bG$
has the induced PBW-filtration from $U(\frg)$ and we may form
$$\Uinvm :=\sum_{i\geq 0} \pi_L^{in}F_iU(\frg)^\bG.$$ It is a subalgebra of
$U(\frg)_n$. Let $\frn$ and $\frn^-$ be the Lie algebras of $\bN$
and $\bN^-$ respectively. The triangular decomposition
$\frg=\frn^-\oplus\frt\oplus\frn$ induces the linear projection
$U(\frg)\ra U(\frt)$ with kernel $\frn^{-}U(\frg)+U(\frg)\frn$. As
in characteristic zero one shows that its restriction to the zero
weight space with respect to the adjoint action of $\bT$ on
$U(\frg)$ is an algebra homomorphism. We therefore have an algebra
homomorphism
$$\phi: U(\frg)^\bG\ra U(\frt)$$ (cf. also \cite{JantzenCharp}, \S9 for the corresponding construction over a field of positive
characteristic). It extends to a homomorphism
\begin{numequation}\label{equ-phi}\phi_n: \Uinvm\lra \Utm.\end{numequation} For all that follows we fix a linear homomorphism
$$\lambda\in\Hom_{o_L}(\pi_L^n\frt,o_K)$$ and view it as a character
$U(\frt)_n\ra o_K$. Using $\phi_n$ we may form the algebra
$$\Ulm:=\Ugm\otimes_{\Uinvm,\lambda}o_K$$
and its $p$-adic completion $\hUlm$. We let
$\hUlmK:=\hUlm\otimes_{o_K} K$. Since $U(\pi_L^n\frg)$ is left and
right noetherian, the rings $\hUlm$ and $\hUlmK$ are left and
right noetherian (\cite{BerthelotDI}, (3.2.2) (iii)).

\vskip8pt

Remark: In \cite{AW} the authors consider only the case
$\lambda(\pi_L^n\frt)\subseteq o_L$. However, all constructions of
loc.cit. immediately generalize to the slightly more general case
considered here.

\end{para}

\begin{para}
Recall the relative universal enveloping algebra $\tiD$ of the
torsor $\xi$. Let $U\in\cS$ and choose a trivialization
$\xi^{-1}(U)\simeq U\times_{o_L}\bT$. The induced isomorphism
\begin{numequation}\label{equ-Dtrivial}\tiD(U)\simeq \cD_X(U)\otimes_{o_L} U(\frt)\end{numequation} translates the
positive $\bbZ$-filtration on the source into the tensor product
filtration on the target (where the second factor in the target
has its usual PBW-filtration). Thus, $\tiD(U)$ is a deformable
algebra and hence, there is a $n$-th deformation
\begin{numequation}\label{equ-deform}\tiD(U)_n :=\sum_{i\geq 0} \pi_L^{in}F_i\tiD(U).\end{numequation} We
remark in passing that (\ref{equ-Dtrivial}) induces an algebra
isomorphism
\begin{numequation}\label{equ-def}\tiD(U)_n\simeq
\cD_X(U)_n\otimes_{o_L} U(\frt)_n.\end{numequation}

The formation $U\mapsto \tiD(U)_n$ yields a presheaf on the basis
$\cS$ of $X$. The sheafification functor produces therefore a
sheaf $\tiD_n$ on $X$ whose algebra of sections over any $U\in\cS$
coincides with $\tiD(U)_n$ (e.g. \cite{BGR}, last remarks in
9.2.1). %\footnote{This intermediate step in the formation of
%$\tiD_n$ is not mentioned in \cite{AW}.}
Similarly, there is a
sheaf $\widehat{\tiD_n}$ on $X$ whose algebra of sections over any
$U\in\cS$ equals the $p$-adic completion of $\tiD(U)_n$, i.e.
$$\widehat{\tiD_n}(U)\simeq \widehat{\cD_X(U)_n}\hat{\otimes}_{o_L}
\widehat{U(\frt)_n}.$$ The formation of the sheaf
$\widehat{\tiD_n}$ is compatible with the homomorphisms $\varphi$
and $\psi$ (cf. (\ref{equ-varphi}) and (\ref{equ-psi})) whence two
homomorphisms
\[\begin{array}{ccccc}
  \widehat{\varphi_n}: \hUgm\lra \widehat{\tiDm} &  &{\rm and}  && \widehat{\psi_n}: \hUtm\lra \widehat{\tiDm}.  \\
\end{array}
\]
The second map remains a central embedding. We therefore may form
the central reduction
$$\hDml:=\widehat{\tiDm}\otimes_{\hUtm,\lambda} o_K.$$ %Given
% we have an isomorphism
%$$\hDml(U)\simeq\widehat{D_X(U)_n}$$ according to the above.

It is a major technical result in \cite{AW} (loc.cit, Lemma 4.10)
that these two maps are related via the $p$-adic completion of the
deformed 'Harish-Chandra homorphism' (\ref{equ-phi})
$$\widehat{\phi_n}: \widehat{\Uinvm}\lra \widehat{\Utm}.$$ This
yields a morphism of sheaves of algebras
$$\widehat{\varphi^\lambda_n}: \hUlm= \widehat{\Ugm}\otimes_{\widehat{\Uinvm},\lambda}o_K \lra
\widehat{\tiDm}\otimes_{\hUtm,\lambda} o_K=\hDml.$$ We denote the
base change to $K$ of these sheaves by $\hDmlK$ and $\hUlmK$
respectively and obtain a morphism
$\widehat{\varphi^\lambda_n}_{,K}:\hUlmK\ra\hDmlK.$ In the
following we will regard $\hDmlK$ as a module over $\hUlmK$ via
this map. Note that $\hDmlK$ is a sheaf of coherent rings
(loc.cit., Prop. 6.5 (d)).

\end{para}

\begin{para}
%The sheaf $\hDml$ is a coherent sheaf of rings and
We have the abelian categories ${\rm coh}(\hDmlK)$ and ${\rm
coh}(\Gamma(\hDmlK))$ of coherent (left) $\hDmlK$-modules
respectively $\Gamma(\hDmlK)$-modules. Given a
$\Gamma(\hDmlK)$-module $M$ we may form the $\hDmlK$-module
$$ {\rm Loc}_{\lambda} (M):=\hDmlK\otimes_{\Gamma(\hDmlK)} M.$$
Conversely, the module of global sections $\Gamma(\cM)$ of a
$\hDmlK$-module $\cM$ is of course a $\Gamma(\hDmlK)$-module. The
functors $({\rm Loc}_\lambda, \Gamma(\cdot))$ form an adjoint
pair. The following result may be viewed as a $p$-adically
completed version of the {\it th\'eor\`eme principal} of
\cite{BB81}, restricted to coherent modules. It is one of the main
results of \cite{AW}. The case $n=0$ was already obtained in an
earlier paper by Noot-Huyghe (\cite{NootHuyghe09}). Recall that we
have fixed a homomorphism $\lambda\in\Hom_{o_L}(\pi_L^n\frt,o_K)$.
\begin{thm}\label{thm-BB} {\rm (Ardakov-Wadsley, Noot-Huyghe)} Suppose the weight $\lambda+\rho\in\frt^*_K$ is dominant
and regular. The adjoint pair $({\rm Loc}_\lambda, \Gamma(\cdot))$
induces mutually inverse equivalences of categories
$${\rm coh}(\Gamma(\hDmlK))\car{\rm coh}(\hDmlK).$$
\end{thm}
\end{para}
\begin{para}
Let $\bar{\kappa}$ be an algebraic closure of $\kappa$. Let us
consider the following three hypothesis on the geometric closed
fibre $\bG_{\bar{s}}=\bG\otimes_{o_L}\bar{\kappa}$ of $\bG$ which
are familiar from the theory of modular Lie algebras (cf.
\cite{JantzenCharp}, 6.3).

\begin{itemize}
    \item[(H1)] The derived group of $\bG_{\bar{s}}$ is (semisimple) simply
    connected.
    \item[(H2)] The prime $p$ is good for the $\bar{\kappa}$-Lie algebra $Lie(\bG_{\bar{s}})$.
    \item[(H3)] There exists a $\bG_{\bar{s}}$-invariant non-degenerate bilinear form on $Lie(\bG_{\bar{s}})$.
\end{itemize}

A prominent example satisfying these conditions for all primes $p$
is the general linear group (using the trace form for {\rm (H3)}).
Any almost simple and simply connected $\bG_{\bar{s}}$ satisfies
these conditions if $p\geq 7$ (and if $p$ does not divide $n+1$ in
case $\bG_{\bar{s}}$ is of type $A_n$). For a more detailed
discussion of these conditions we refer to loc.cit.

\vskip8pt

The next theorem is a version of \cite{AW}, Thm. 6.10. It is
proved in loc.cit. under the assumptions that $\bG$ is semisimple
and simply connected and that the prime $p$ is {\it very good} (in
the sense of loc.cit., 6.8).

\begin{thm}\label{thm-globalsections}{\rm (Ardakov-Wadsley)}
Assuming {\rm (H1)-(H3)} the map
$\widehat{\varphi^\lambda_n}_{,K}$ induces an algebra isomorphism
\begin{numequation}\label{equ-varphilambda}\widehat{\varphi^\lambda_n}_{,K}:\hUlmK\car\Gamma(X,\hDmlK).\end{numequation}
\end{thm}
\begin{proof}
Let $\frt_{\bbZ}$ denote the Lie algebra of the canonical
extension of the torus $\bT$ to a group scheme over $\bbZ$, and
let $\rho$ be the map $\Phi^\vee\car\Phi\subset \frt^\vee_{\bbZ}$
between coroots, roots and the dual of the free $\bbZ$-module
$\frt_{\bbZ}$ which comes from the root datum of $\bG$. The triple
$\sR=(\frt_{\bbZ},\Phi^\vee, \rho)$ is then a {\it syst\`eme de
racines {\rm (r\'eduit)} pr\'ecis\'e} in the sense of
\cite{DemazureInvariants}. The assumptions in \cite{AW} are made
in order to be able to apply \cite{DemazureInvariants}, \S6.
Corollaire du Th\'eor\`eme 2 and Th\'eor\`eme 3 to $\sR$ and the
commutative ring $\bbF_p$ as well as to be able to apply
\cite{BMR08}, Prop. 3.4.1. However, these results are also
available under the hypotheses (H1)-(H3). Indeed, let $t_{\sR}$ be
the torsion index of $\sR$ (\cite{DemazureInvariants}, \S5) and
assume {\rm (H1)-(H3)}. The discussion in \cite{JantzenCharp}, 9.6
shows that $t_{\sR}$ is invertible in $\bbF_p$ and that the
existence of an $\alpha^\vee\in\Phi^\vee$ with
$\alpha^\vee/2\in\frt_{\bbZ}$ implies $p\neq 2$. Hence,
\cite{DemazureInvariants}, Corollaire du Th\'eor\`eme 2 and
Th\'eor\`eme 3 apply to $\sR$ and the commutative ring $\bbF_p$.
Furthermore, the required statement from \cite{BMR08} extends to
the case of $\bG$ since it depends only on the derived group of
$\bG_{\bar{s}}$. Using these inputs the argumentation in
\cite{AW}, 6.9/10 goes through.
\end{proof}
Remark: Under the hypotheses (H1)-(H3) and $\lambda+\rho$ being
dominant and regular, the isomorphism (\ref{equ-varphilambda})
together with Thm. \ref{thm-BB} implies an equivalence of
categories ${\rm coh}(\hDmlK)\simeq {\rm Mod}^{\rm fg}(\hUlmK)$
where the right-hand side denotes the finitely generated (left)
modules over the noetherian ring $\hUlmK$.

\end{para}

\begin{para}
We now study how the family of sheaves $\hDmlK$ on $X$ varies in
the deformation parameter $n\geq 0$. Let $U\in\cS$. The inclusions
$$F_i\tiD_{n+1}(U)=\pi_L^{i(n+1)}F_i\tiD(U)\subseteq \pi_L^{in}F_i\tiD(U)=F_i\tiD_n(U)$$
for all $i\geq 0$ induce a morphism of sheaves of algebras
$\tilde{\cD}_{n+1}\ra\tilde{\cD}_{n}$ for all $n\geq0$. It extends
to a morphism
\begin{numequation}\label{equ-resD}res_\cD: \widehat{\cD^\lambda_{n+1}}_{,K}\lra\hDmlK\end{numequation} and yields a
projective system $(\hDmlK)_n, n\geq 0$.
\begin{prop}\label{prop-verifyD}
Let $U,V\in\cS$ with $V\subset U$ and $n\geq 0$.

\begin{itemize}
    \item[(i)] Each $\Gamma(U,\hDmlK)$ is a left
    and right noetherian Banach algebra,
    \item[(ii)] the transition homomorphism
$\Gamma(U,\widehat{\cD^\lambda_{n+1}}_{,K})\ra\Gamma(U,\hDmlK)$ is
left and right flat with dense image,

\item[(iii)] the restriction homomorphism
$\Gamma(U,\hDmlK)\ra\Gamma(V,\hDmlK)$ is left and right flat.
\end{itemize}

%In other words, the projective system $(\hDmlK)_n$ satisfies the
%properties (i) and (ii) of section \ref{sect-crossedsheaves}.
\end{prop}
\begin{proof}
Being a $p$-adic completion the ring $\Gamma(U,\hDml)$ is
$p$-adically complete and separated. According to
\cite{BerthelotDI}, (3.2.3) (iv),(vi) it is flat over $o_K$ and
left and right noetherian. In particuar, it is a lattice in the
$K$-vector space $\Gamma(U,\hDmlK)=\Gamma(U,\hDml)\otimes_{o_K} K$
and the corresponding gauge norm (\cite{NFA}, \S2) makes the
latter a (left and right noetherian) $K$-Banach algebra. This
shows (i). The flatness property (iii) follows from the proof of
\cite{AW}, Prop. 5.9 (c) together with the last part of loc.cit.,
Prop. 6.5 (a). Let us establish property (ii). Since $X$ is smooth
we may assume, by passing to a smaller $U\in\cS$, that the tangent
sheaf $\cT_X|_U$ is a free $\cO_X|_U$-module. We choose local
coordinates $t_1,...,t_r$ on $U$ and let
$\partial_1,...,\partial_r$ be the corresponding derivations.
According to (\ref{equ-gradedunivenvelope}) we have an algebra
isomorphism
\[\begin{array}{ccc}
  gr_\bullet \cD_X(U)\car \cO_X(U)[T_1...,T_r], &  & {\rm principal~
symbol~of~}\partial_i\mapsto T_i \\
\end{array}\]
onto the polynomial ring in the variables $T_i$ with coefficients
in $\cO_X(U)$. Now consider the sheaf
$$\Dml:=\tiDm\otimes_{U(\frt)_n,\lambda} o_K$$ on $X$. It follows from
\cite{AW}, Lemma 6.5 that the algebra of sections $\hDml(U)$ over
an open set $U\in\cS$ equals the $p$-adic completion of $\Dml(U)$.
By general results on adic completions of non-commutative rings
(e.g \cite{BerthelotDI}, (3.2.3) (vii)) we are therefore reduced
to prove that the homomorphism $\cD^\lambda_{n+1}(U)\ra
\cD^\lambda_n(U)$ is flat. To do this choose a trivialization
$$\xi^{-1}(U)\simeq U\times_{o_L}\bT$$ of the torsor $\xi$ over
$U$. The isomorphism (\ref{equ-def}) is compatible with variation
in $n$. Applying the toral character $\lambda$ we are reduced to
prove that
$$\cD_X(U)_{n+1}\subseteq \cD_X(U)_n\hskip25pt {\rm (*)}$$ is a flat homomorphism. Now
$\cD_X(U)_m\subseteq \cD_X(U)$ has the subspace filtration for all
$m\geq 0$ and the inclusion ${\rm (*)}$ is therefore a filtered
morphism. The induced $PBW$-filtration on source and target is
positive, hence complete. By general principles (e.g. \cite{ST5},
Prop. 1.2) we are therefore reduced to prove flatness of the
graded map
$$gr_\bullet \cD_{n+1}(U)\subseteq gr_\bullet \cD_n(U).$$ By choice of
$t_1,..,t_r$ this map equals the natural inclusion
$$\cO_X(U)[\pi_L^{n+1}T_1,...,\pi_L^{n+1}T_r] \subseteq
\cO_X(U)[\pi_L^{n}T_1,...,\pi_L^{n}T_r].$$ Since source and target
are flat over the discrete valuation ring $o_L$ the assertion
follows by a straightforward application of the fiber criterion
for flatness (\cite{EGA_IV_3}, 11.3.10).
%It is isomorphic to the inclusion
%$$\cO_X(U)[\pi T_1,...,\pi T_r] \subseteq
%\cO_X(U)[T_1,...,T_r].$$ This map is flat over the generic point and
%over the closed point of $Spec(o_L)$ and therefore flat.
To prove the statement about the image we consider the sheaf
$$\cD^\lambda_K:=\cD_0^\lambda\otimes_{o_K} K=\tilde{\cD}\otimes_{U(\frt_K),\lambda_K} K.$$
It is obvious from (\ref{equ-deform}) that
$\tilde{\cD}\otimes_{o_K} K=\tilde{\cD}_n\otimes_{o_K}K$ and thus
$\cD^\lambda_K=\cD^\lambda_n\otimes_{o_K} K$ holds for all $n\geq
0$. We deduce that the canonical homomorphism
$$\cD^\lambda_K(U)\lra\hDmlK(U)$$ has dense image for all $n$ and is compatible with
variation in $n$. This proves the assertion.
\end{proof}
\begin{cor}\label{cor-conditions1}
The projective system of sheaves $(\hDmlK)_{n\geq 0}$ on $X$
satisfies the assumptions {\rm (i)},{\rm (ii)},{\rm (iii)} of
section \ref{sect-crossed}.
\end{cor}

\vskip8pt

We have the inclusions $U(\frg)_{n+1}\subseteq U(\frg)_n$ and
similarly for the algebras $U(\frg)^\bG$ and $U(\frt)$. The maps
$\varphi, \phi$ and $\psi$ are visibly compatible with these
inclusions. %Using the same methods as above together with the
%computations in the proof of \cite{AW}, Thm. 6.8 one may also show
%that the induced homomorphism
%$$\widehat{\cU^\lambda_{n+1}}_{,K}\lra \widehat{\cU^\lambda_{n}}_{,K}$$
%is left and right flat with dense image for all $n$.
We therefore obtain an algebra homomorphism
$$res_{u}:\widehat{\cU^\lambda_{n+1}}_{,K}\lra
\widehat{\cU^\lambda_{n}}_{,K}$$ which fits into the commutative
diagram
\[\xymatrix{
\widehat{\cU^\lambda_{n+1}}_{,K}\ar[d]^{res_{u}}\ar[r]^{\widehat{\varphi^\lambda_{n+1}}_{,K}} & \widehat{\cD^\lambda_{n+1}}_{,K} \ar[d]^{res_\cD}\\
 \widehat{\cU^\lambda_{n}}_{,K} \ar[r]^{\widehat{\varphi^\lambda_{n}}_{,K}}  &\hDmlK. }
\]

\end{para}

\section{Locally analytic
representations}\label{sect-dis}

\begin{para}
We recall some definitions and results about distribution algebras
of compact locally $L$-analytic groups (\cite{ST4}, \cite{ST5}).
We consider a locally $L$-analytic group $G$ and denote by
$C^{an}(G,K)$ the locally convex $K$-vector space of locally
$L$-analytic functions on $G$ as defined in \cite{ST4}. The strong
dual
$$D(G):=D(G,K):=C^{an}(G,K)'_b$$ is the algebra of $K$-valued
locally analytic distributions on $G$. The multiplication
$\delta_1\cdot\delta_2$ of distributions $\delta_1,\delta_2\in
D(G)$ is given by the convolution product
\[\delta_1\cdot\delta_2(f) =\delta_2(h_2\mapsto
\delta_1(h_1\mapsto f(h_1h_2))).\] Since $G$ is compact, $D(G)$ is
a $K$-Fr\'echet algebra.

\end{para}

\begin{para}

The universal enveloping algebra $U(\frg)$ of the Lie algebra
$\frg:=Lie(G)$ of $G$ acts naturally on $C^{an}(G,K)$. On elements
$\frx\in\frg$ this action is given by \[(\frx f)(h) = \frac{d}{dt}
(t\mapsto f(\exp_G(-t\frx)h))|_{t=0}\] where $\exp_G:\frg-->G$
denotes the exponential map of $G$, defined in a small
neighbourhood of $0$ in $\frg$. This gives rise to an embedding of
$U(\frg)_K:= U(\frg)\otimes_L K$ into $D(G)$ via

\[ U(\frg)_K\hookrightarrow D(G),~~\frx\mapsto
(f\mapsto(\dot{\frx}f)(1)).\]

Here $\frx\mapsto\dot{\frx}$ is the unique anti-automorphism of
the $K$-algebra $U(\frg)_K$ which induces multiplication by $-1$
on $\frg$.

\end{para}
\begin{para}

We will occasionally write $G_0$ for the same group $G$ but with
the locally $\bbQ_p$-analytic structure induced by restriction of
scalars. We point out that there is a canonical isomorphism of
$\bbQ_p$-Lie algebras $Lie(G)\simeq Lie(G_0)$ (\cite{B-VAR},
5.14.5) which we will use to identify both $\bbQ_p$-algebras from
now on. The inclusion of locally $L$-analytic functions into
locally $\bbQ_p$-analytic functions on $G$ gives rise to a
quotient map of topological algebras $D(G_0)\rightarrow D(G)$
whose kernel ideal is generated by the kernel of the natural map
\begin{numequation}\label{scalar} L \otimes_{\bbQ_p} Lie(G_0)\longrightarrow Lie(G),
a\otimes\frx\mapsto a\frx\end{numequation} (\cite{SchmidtAUS},
Lemma 5.1).

\end{para}

\begin{para}

We will now specialize to uniform pro-$p$ groups. We refer to
\cite{DDMS} for an extensive study of this important class of
locally $\bbQ_p$-analytic groups. We assume for the rest of this
section that $p\neq 2$. So let $G$ be a uniform pro-$p$ group of
dimension $d:={\rm dim}_{\bbQ_p} G$. Let $h_1,...,h_d$ be a set of
topological generators. In particular, any element $g\in G$ can
uniquely be written as $g=h_1^{\lambda_1}\cdot\cdot\cdot
h_d^{\lambda_d}$ with $\lambda_1,...,\lambda_d\in\bbZ_p$.
According to [loc.cit.], Thm. 4.30 the operations
\[\begin{array}{ccl}
  \lambda\cdot x & = & x^\lambda,\\
  &  &  \\
  x+y & = & \limi_{i\rightarrow\infty} (x^{p^{i}}y^{p^{i}})^{p^{-i}}, \\
   &  &  \\
 {[x,y]}  & = & \limi_{i\rightarrow\infty}  (x^{-p^{i}}y^{-p^{i}}x^{p^{i}}y^{p^{i}})^{p^{-2i}} \\
\end{array}
\]
for $x,y\in G$ and $\lambda\in\bbZ_p$ define on the set $G$ the
structure of a Lie algebra $L_G$ over $\bbZ_p$. This Lie algebra
is {\it powerful} in the sense that it is a free $\bbZ_p$-module
(of finite rank $d$) and satisfies $[L_G,L_G]\subseteq pL_G$.

\vskip8pt

The logarithm map $\log_G$ of the Lie group $G$ induces an
injective homomorphism of $\bbZ_p$-Lie algebras
$L_G\hookrightarrow Lie(G)$ (\cite{DDMS}, sect. 4.5/9.4). We will
therefore view $L_G$ as a distinguished $\bbZ_p$-lattice in the
$\bbQ_p$-Lie algebra $Lie(G)$.
%In other words, $\exp_G$ is defined on $L_G\subset Lie(G)$ and
%maps it bijectively onto $G$. In turn, the $\bbZ_p$-lattice
%$L_G\subset Lie(G)$ is uniquely determined by these two latter
%properties (\cite{DDMS}, sect. 4.5/9.4).

\end{para}

\begin{para} Now let $G$ be a compact locally $L$-analytic group of
dimension $d={\rm dim}_L G$. Following
\cite{SchmidtAUS},\cite{OrlikStrauchIRR} we call $G$ {\it
$L$-uniform} if
\begin{itemize}
    \item[(1)] $G$ is uniform pro-$p$,
    \item[(2)] the $\bbZ_p$-submodule $L_{G}\subset Lie(G)$ is stable under
    multiplication by $o_L$ (and so is an $o_L$-submodule of $Lie(G)$).
\end{itemize}
We remark that any compact locally $L$-analytic group $G$ has a
basis of neighbourhoods of $1\in G$ consisting of open normal
subgroups which are $L$-uniform (\cite{SchmidtAUS}, Cor.4.4).

\vskip8pt

Let in the following $G$ be a locally $L$-analytic group of
dimension $d$ which is $L$-uniform. Then $L_{G}$ is a free
$o_L$-module of rank $d$. Let $\frx_1,...\frx_d$ be an $\R$-basis
of $L_{G}$ and let $v_1,...,v_{[L:\bbQ_p]}$ be a $\bbZ_p$-basis of
$\R$ with the additional property $v_1=1$. Then the elements
\[h_{ij}:=\exp_G(v_i\frx_j)\] are a set of topological
generators of $G$ of cardinality $[L:\bbQ_p]d={\rm dim}_{\bbQ_p}
G_0$.

\end{para}

\begin{para}

\vskip8pt

The underlying $K$-vector space of $D(G_0)$ admits the following
description. Fix an ordering of the generators $h_{ij}$. Let
$b_{ij}:=h_{ij}-1\in\bbZ[G]$ and $\bb^\al:=\prod_{ij}
b_{ij}^{\al_{ij}}\in\bbZ[G]$ (in the given ordering) for
$\al\in\bbN_0^{[L:\bbQ_p]d}$. Then $D(G_0)$ equals the set of all
convergent series
\[\lambda=\sum_{\al\in\bbN_0^{[L:\bbQ_p]d}}d_\al\bb^\al\]
with $d_\al\in K$ such that the set $\{|d_\al|r^{|\al|}\}_\al$ is
bounded for all real numbers $0<r<1$. The inclusion
$Lie(G_0)\subset D(G_0)$ is given on basis elements by
\begin{numequation}\label{compLie} v_i\frx_j\mapsto\log
(1+b_{ij})\end{numequation} where
$\log(1+X)=X-X^2/2+X^3/3-\cdot\cdot\cdot$ is the usual logarithm
series. The family of norms $||.||_r,~0<r<1$ where
\[ ||\lambda||_r:=\sup_\al |d_\al|r^{|\al|}\]
and their quotient norms define the Fr\'echet topologies on
$D(G_0)$ respectively $D(G)$. The norms belonging to the interval
$p^{-1}\leq r<1$ where $r\in p^{\bbQ}$ are multiplicative. In this
case the corresponding norm completions $D_r(G_0)$ and $D_r(G)$ of
$D(G_0)$ and $D(G)$ respectively are $K$-Banach algebras.
Furthermore, we emphasize that the norms in this case do not
depend on the choice of minimal set of topological generators for
the group $G$ (\cite{ST5}, discussion after Thm. 4.10).

\end{para}

\begin{para}

Let $r\in p^{\bbQ}$ such that $p^{-1}\leq r<1$. There is the
following norm filtration on $D_{r}(G_0)$ defined by the additive
subgroups
\[\begin{array}{rl}
  F^s_rD_r(G_0):= & \{\lambda\in D_r(G_0),~||\lambda||_r\leq p^{-s}\}, \\
   &  \\
  F^{s+}_r D_r(G_0):= & \{\lambda\in D_r(G_0),~||\lambda||_r < p^{-s}\}
  \\
\end{array}
\] for $s\in\bbR$ with graded ring
\[\gr D_r(G_0):=\oplus_{s\in\bbR}~gr_r^s D_r(G_0)\]
where $gr_r^s D_r(G_0):=F^s_r D_r(G_0)/F^{s+}_r D_r(G_0)$.

\vskip8pt

We may use the quotient norm of $||.||_r$ to define a norm
filtration on $D_r(G)$ as above. It coincides with the quotient
filtration of $F_r^\bullet D_r(G_0)$. We denote the associated
graded ring by $\gr D_r(G)$.

 Similarly, the absolute value on $K$
induces a filtration and a graded ring $gr K$ as above. The choice
of uniformizer $\pi_K$ induces an isomorphism of $gr K$ with the
ring of Laurent polynomials in one variable over $k$. The rings
$\gr D_r(G_0)$ and $\gr D_r(G)$ are naturally $gr K$-algebras.

\vskip8pt

Now let additionally $p^{-1}<r<1$. Then Thm. 4.5 of \cite{ST5}
shows that each $\gr D_r(G_0)$ equals a polynomial ring over $gr
K$ in $[L:\bbQ_p]d$ variables (the principal symbols of the
$b_{ij})$. Hence any $\gr D_r(G)$ will be a commutative $gr
K$-algebra.

\end{para}

\begin{para}

Let $\{G^{p^m}\}_{m\geq 0}$ be the lower $p$-series of the uniform
pro-$p$ group $G$, i.e.
\[G^{p^m}:=\{g^{p^m}: g\in G\} \]
for $m\geq 0$. Each $G^{p^m}$ is a characteristic open uniform
pro-$p$ subgroup of $G$ of index $p^{[L:\bbQ_p]md}$. The set
$\{G^{p^m}\}_m$ constitutes a neighbourhood basis of $1\in G$.
Because of $L_{G_0^{p^m}}=p^mL_{G_0}$ each open subgroup
$G^{p^m}$, endowed with its induced locally $L$-analytic
structure, is $L$-uniform. By functoriality of $D(\cdot)$
(\cite{KohlhaaseI}, 1.1) we obtain a series of subalgebras
\[D(G_0)\supset D(G_0^{p})\supset D(G_0^{p^2})\supset...\]
which yields, by passage to quotients, a series
\[D(G)\supset D(G^{p})\supset D(G^{p^2})\supset...\]
of subalgebras of $D(G)$. For each $m\geq 0$ we may apply the
above discussion to the group $\Gpm$ and obtain a family of norms
$||.||_r^{(m)}$ on $D(\Gzpm)$ as well as quotient norms $q^m_{r}$
on $D(\Gpm)$. Of course, these norms again induce filtrations and
graded rings as explained above whenever $r\in p^{\bbQ}$ with
$p^{-1}\leq r <1$.

%\begin{lemma}\label{lem1.1}
%Let $0<r'<1$ in $p^{\bbQ}$ and $m\geq 0$ such that
%$r:=\sqrt[p^m]{r'}\geq p^{-1}$. Then $||._r^{(0)}$ on $D(G_0)$
%restricts to $||.||_{r'}^{(m)}$ on the subring $D(\Gzpm)\subset
%D(G_0)$.
%\end{lemma}
%\begin{proof}
%This is proved in \cite{SchAUS}, Prop. 6.2 with the restriction
%$r'>p^{-1}$. The same proof gives this more general result.
%\end{proof}

\vskip8pt

To ease notation we put for $m\geq 0$ from now on
\[r_m:=\sqrt[p^m]{1/p}.\]
In particular, $r_0=1/p$.
\begin{lemma}\label{lem-free}
The ring extension $D(\Gpm)\subset D(G)$ completes to a ring
extension
\[D_{\rn}(\Gpm)\subset D_{r_m}(G)\] and $D_{r_m}(G)$ is a finite
and free (left or right) module over $D_{\rn}(\Gpm)$ on a basis
any system of coset representatives for the finite group $G/\Gpm$.
\end{lemma}
\begin{proof}
This is a slight generalization of \cite{SchmidtAUS}, Lemma 7.4
using the fact that the norm $||.||_{r_m}^{(0)}$ on $D(G_0)$
restricts to $||.||_{\rn}^{(m)}$ on the subring $D(\Gzpm)\subset
D(G_0)$. This latter fact is itself a mild extension of
[loc.cit.], Prop. 6.2 which follows with the same proof given
there.
\end{proof}
\end{para}
\begin{para}\label{para-TAU}

Completely similar to the standard example of the group ring
(\cite{Passman}) the ring $D(G)$ is a crossed product of the ring
$S:=D(\Gpm)$ by the group $H:=G/\Gpm$ with action and twisting
given as follows. Fix a set of representatives
$\overline{H}\subset G$ for the group $G/\Gpm$ and view it in
$D(G)^\times$ via the map $\delta$ (associated Dirac
distribution). Then
$$\sigma(h)(s)=\delta_{\bar{h}}^{-1}s\delta_{\bar{h}}$$
for $s\in S$. On the other hand, given $h_1,h_2\in H$ we have
\[\bar{h}_1\bar{h}_2=\overline{h_1h_2}\cdot\tau(h_1,h_2)\]
with some uniquely determined $\tau(h_1,h_2)\in\Gpm$ and the
resulting $\tau:H\times H\rightarrow\Gpm\subseteq S^\times$ equals
the twisting.

\vskip8pt

As we have pointed out the norm $||.||_{\rn}^{(m)}$ on $S=D(\Gpm)$
is independent of the choice of topological generators for the
group $\Gpm$. This means that for $g\in G$ the conjugation
automorphism on $S=D(\Gpm)$ given by
$x\mapsto\delta_g^{-1}x\delta_g$ is an isometry. Thus, $\sigma$
extends to an action $H\rightarrow {\rm Out} (D_{\rn}(\Gpm))$.
Likewise $\tau$ extends to a $2$-cocycle with values in
$D_{\rn}(\Gpm)^\times$ and we may form the crossed product
\[D_{\rn}(\Gpm)\ast_{\sigma,\tau} H.\]

\begin{cor}\label{cor-crossed} There is a natural $K$-algebra isomorphism
\[D_{\rn}(\Gpm)\ast_{\sigma,\tau} G/\Gpm\car D_{r_m}(G).\]
\end{cor}
\begin{proof}
The preceding lemma gives a $K$-algebra homomorphism
$D_{\rn}(\Gpm)\ast_{\sigma,\tau} G/\Gpm\rightarrow D_{r_m}(G)$
induced from the inclusion $D_{r_0}(\Gpm)\subset D_{r_m}(G)$ which
is bijective.

\end{proof}

Remark: In the case $L=\bbQ_p$ this responds to an issue raised in
\cite{AW}, 1.5.

\end{para}

\begin{para}
We fix $m\geq 1$ and consider the filtration on the Banach algebra
$D_{\rn}(G^{p^m})$ induced by the quotient norm $q^m_{r_0}$ from
5.10. The $[L:\bbQ_p]d$ elements $h_{ij}^{p^m}$ are a minimal
ordered set of topological generators for the group $G^{p^m}$.
Define the $d$ elements $c_j:=h_{1j}^{p^m}-1$ and let
$\bc^{\beta}=\prod_j c_j^{\beta_j}$ for $\beta\in\bbN^d$.

\begin{lemma}\label{lem-graded}
Let $m\geq 1$. The graded ring $gr^\bullet_{r_0}D_{r_0}(\Gpm)$
equals a polynomial algebra over $gr K$ in the principal symbols
of the $d$ elements $c_j$. The $K$-vector space underlying
$D_{\rn}(\Gpm)$ is given by all series
\[\lambda=\sum_{\beta\in\bbN_0^d}d_\beta\bc^\beta\]
with $|d_\beta|{\rn}^{|\beta|}\rightarrow 0$ for
$|\beta|\rightarrow\infty$. Moreover, the norm
$q^m_{r_0}(\lambda)$ may be computed as
\[q^m_{r_0}(\lambda)=\sup_\beta q^m_{r_0}(d_\beta\bc{^\beta})=\sup_\beta|d_\beta|{\rn}^{|\beta|}.\]
\end{lemma}
\begin{proof}
We have already remarked (proof of lemma \ref{lem-free}) that
$||.||_{r_m}^{(0)}$ restricts to $||.||_{\rn}^{(m)}$ on the
subring $D(\Gzpm)\subset D(G_0)$ from which we obtain an inclusion
of rings
\[ gr_{\rn}D_{\rn}(\Gzpm)\subset gr_{r_m}D_{r_m}(G_0)\]
by left exactness of the functor $gr$. Since $m\geq 1$ we have
$p^{-1}<r_m$. Thus the right-hand side is a commutative ring. It
follows that $gr_{\rn}D_{\rn}(\Gzpm)$ is a polynomial algebra over
$gr K$ in the $nd$ principal symbols of the $h_{ij}^{p^m}-1$.
Arguing as in \cite{SchmidtAUS}, Prop. 5.6/9 the quotient
$gr_{r_0}D_{\rn}(\Gpm)$ equals a polynomial algebra in the
principal symbols of the $c_j$ and the remaining statements follow
from this.
%A slight generalization of
%\cite{SchmidtAUS}, Prop. 5.9 shows that

\end{proof}
Remark: The result generally does not hold for $m=0$: let
$L=\bbQ_p$ and
$$\bbZ_p[[G]]=\varprojlim_N \bbZ_p[G/N]$$ be the completed
group ring of the profinite group $G$. Here, $N$ varies over the
open normal subgroups of $G$. There is an embedding
$\bbZ_p[[G]]\hookrightarrow D_{\rn}(G)$. By one of the main
theorems of M. Lazard's work on $p$-adic analytic groups
(\cite{LazardGroupesAnalytiques}, Thm. III.2.3.3) the algebra
$gr_{\rn}\bbZ_p[[G]]$ equals the universal enveloping algebra of a
certain Lie algebra associated to the group $G$ which, generally,
is non-abelian.

\vskip8pt

Still assuming $m\geq 1$ the preceding discussion shows that the
$o_K$-submodule $F_{\rn}^0D_{\rn}(\Gpm)$ of $D_{\rn}(\Gpm)$
consists of all series
\[\lambda=\sum_{\beta\in\bbN_0^d}e_\beta (c_1/p)^{\beta_1}\cdot\cdot\cdot(c_d/p)^{\beta_d}\]
with $|e_\beta|\leq 1$ and $|e_\beta|\rightarrow 0$ for
$|\beta|\rightarrow\infty$. Mapping such a series to $\sum_\beta
(e_\beta~{\rm mod}~\pi_K) u_1^{\beta_1}\cdot\cdot\cdot
u_d^{\beta_d}$ induces an isomorphism of $k$-algebras between

\[gr^0_{\rn}
D_{\rn}(\Gpm)=F_{\rn}^0D_{\rn}(\Gpm)/\pi_KF_{\rn}^0D_{\rn}(\Gpm)\]

and the polynomial algebra over $k$ in the variables
$u_1,...,u_d$. We see that the $o_K$-subalgebra
$F_{\rn}^0D_{\rn}(\Gpm)$ is a $\pi_K$-adically separated and
complete lattice in the $K$-algebra $ D_{\rn}(\Gpm)$. Moreover,
the $\pi_K$-adic reduction of this lattice is endowed with a
positive and hence complete split $\bbZ$-filtration given by total
degree with respect to the variables $u_j$. We conclude that
$D_{\rn}(\Gpm)$ is a {\it complete doubly filtered $K$-algebra} in
the sense of \cite{AW}, Def. 3.1 with
\begin{numequation}\label{equ-Gr} Gr(D_{\rn}(\Gpm)):=gr (gr^0_{\rn}
D_{\rn}(\Gpm))\simeq k[u_1,...,u_d].\end{numequation}
\end{para}

\begin{para}
Let $G$ be a pro-$p$ group which is $L$-uniform. The projective
system of noetherian $K$-Banach algebras $D_{r_m}(G)$ for $m\geq
1$ defines the so-called {\it Fr\'echet-Stein structure} of
$D(G)$. Without recalling the precise definition from \cite{ST5}
this essentially means that the natural map

\[D(G)\car\varprojlim_{m\geq 1} D_{r_m}(G)\]

is an isomorphism of topological $K$-algebras (where the
right-hand side is endowed with the projective limit topology) and
that the transition maps in the projective system are flat ring
homomorphisms. This structure gives rise to a well-behaved abelian
full subcategory $\sC_{G}$ of the (left) $D(G)$-modules, the {\it
coadmissible modules}. By definition, an abstract $D(G)$-module
$M$ is coadmissible if
\begin{itemize}
    \item[(i)] $M_m:=D_{r_m}(G)\otimes_{D(G)}M$ is finitely
    generated over $D_{r_m}(G)$ for all $m\geq 1$;
    \item[(ii)] the natural map $M\car\varprojlim_m M_m$ is an
    isomorphism.
\end{itemize}

The system $\{M_m\}_m$ is sometimes called the {\it coherent
sheaf} associated to $M$. In fact, the projective limit functor
induces an equivalence of categories between the projective
systems $(M_m)_m$ of finitely generated $D_{r_m}(G)$-modules $M_m$
with the property $$D_{r_m}(G)\otimes_{D_{r_{m+1}}(G)} M_{m+1}\car
M_m$$ and $\sC_G.$ To give an example, any $D(G)$-module which is
finitely presented as $D(G)$-module is coadmissible.

\vskip8pt

Remarks: The theory of Fr\'echet-Stein algebras is modelled
according to the example $G=\bbZ_p$, the additive group of
$p$-adic integers. In this case, the Fourier isomorphism of Y.
Amice (\cite{Amice2}) identifies $D(\bbZ_p)$ with the ring of
holomorphic functions on the rigid analytic open unit disc over
$K$. The latter is a quasi-Stein space in the sense of R. Kiehl
(\cite{Kiehl}) and the Fourier isomorphism identifies the category
of coadmissible modules with the coherent module sheaves on the
disc.

\end{para}

\begin{para}
For the definition of an {\it admissible locally analytic
representation} as well as the foundations of the theory of such
representations we refer to \cite{ST5}. However, we at least
recall ([loc.cit.], Thm. 6.3) that the abelian category $Rep(G)$
of admissible locally analytic $G$-representations over $K$ is
(essentially by definition of admissibility) anti-equivalent to
$\sC_{G}$,
$$Rep(G)\car\sC_{G}.$$

The functor is given by sending a representation $V$ to its
so-called {\it strong dual} $M:=V'_b$. As a locally convex
$K$-vector space $V'_b$ equals the continuous dual of the locally
convex vector space $V$ endowed with a certain strong topology
(see \cite{NFA} for all terminology from non-archimedean
functional analysis). The contragredient $G$-representation on
$V'_b$ extends naturally to a $D(G)$-module structure giving the
coadmissible module.
\end{para}

\section{An application of an isomorphism of Lazard}\label{sect-lazard}
In this section we assume $p\neq 2$.
\begin{para} Let $G$ be a uniform pro-$p$ group. Its integral $\bbZ_p$-Lie algebra $L_G$ was first mentioned by M. Lazard
(\cite{Lazard65}, Ex. III. 2.1.10). We require in the following
some information on the relation of $L_G$ with the group ring
$\bbZ_p[G]$ of $G$ and we therefore recall parts of Lazard's
notions and results. We, however, do not recall the standard
definitions of filtrations, valuations etc. (cf. loc.cit., I.1).

\vskip8pt

To start with, the group ring $\bbZ_p[G]$ has a canonical
filtration defined as the lower bound of all filtrations $\omega$
of the $\bbZ_p$-algebra $\bbZ_p[G]$ such that
$$\omega(g-1)\geq \omega(g)$$
for all $g\in G$ (loc.cit., III.2.3.1.2). %Let $$\bbZ_p[[G]]:=\varprojlim_N \bbZ_p[G/N]$$
%be the completed group ring of $G$ (with $N$ running through all
%open normal subgroups of $G$). The natural map
%$\bbZ_p[G]\ra\bbZ_p[[G]]$ identifies the completion of $\bbZ_p[G]$
%with respect to the filtration $w$ with $\bbZ_p[[G]]$.
Since $G$ is a uniform pro-$p$ group $\omega$ is in fact a
valuation of the $\bbZ_p$-module $\bbZ_p[G]$ which allows to
define its so-called {\it saturation} ${\rm Sat~}\bbZ_p[G]$
(loc.cit., I.2.2.11). The latter equals the completion of
$$\{ y\in \bbQ_p[G] |~
\tilde{\omega}(y)\geq 0\}$$ in the topology defined by
$\tilde{\omega}$ and is an associative algebra containing
$\bbZ_p[G].$ Note here that the valuation $\omega$ on $\bbZ_p[G]$
is extended to a map $\tilde{\omega}$ on
$\bbQ_p[G]=\bbQ_p\otimes_{\bbZ_p} \bbZ_p[G]$ by
$$\tilde{\omega}(\lambda^{-1}\otimes m):=\omega(m)-v_p(\lambda)$$
where $v_p(.)$ equals the $p$-adic valuation on $\bbZ_p$.

\vskip8pt

On the other hand, the Lie algebra $L_G$ is a finite free
$\bbZ_p$-module. Fixing a basis and using the valuation $v_p$ on
coefficients it is a valued $\bbZ_p$-module. This valuation
induces a valuation on the universal enveloping algebra $U(L_G)$
(loc.cit., IV. 2.2.1) and the corresponding saturation ${\rm Sat~}
U(L_G)$ does not depend on the choice of valuation of $L_G$. Put
$\frh:=\frac{1}{p}L_G$. This is a $\bbZ_p$-Lie algebra since $L_G$
is powerful. Let
$$\widehat{U(\frh)}=(\varprojlim_i
U(\frh)/p^{i}U(\frh))
$$ be the $p$-adic completion of $U(\frh)$. The inclusion $L_G\subset\frh$ induces an isomorphism
of ${\rm Sat~}U(L_G)$ with $\widehat{U(\frh)}$ which, combined
with the isomorphism loc.cit., IV. 3.2.5, gives an isomorphism
\begin{numequation}\label{equ-Lazardiso}\cL_G: \widehat{U(\frh)}\car{\rm Sat~}\bbZ_p[G].\end{numequation}
By loc.cit., Cor. 3.2.4 it is given on elements $g\in L_G$ via
$\cL_G(g):=Log (g)$ where

$$Log(X)=(X-1)-(X-1)^2/2+(X-3)^3/3-\cdot\cdot\cdot$$ %cf. Lazard III.1.1.5.2 page 72.
is the usual logarithm power series (loc.cit., III.1.1.5.2).
\vskip8pt

Remark: A concise account of the integral Lazard Lie algebra and
its various relations to the group ring can also be found in
\cite{HKN}.

\vskip8pt

Let
\[
\widehat{U(\frh)}_{\bbQ_p}:=\widehat{U(\frh)}\otimes_{\bbZ_p}\bbQ_p.
\]
\begin{prop}\label{prop-Lazard} The map $\cL_G$ and the inclusion $\bbZ_p[G]\subset D(G,\bbQ_p)$ induce an
isomorphism of $\bbQ_p$-Banach algebras
\begin{numequation}\label{equ-Lazard} \cL_G:
\widehat{U(\frh)}_{\bbQ_p}\car
%Q_{gr p}(\bbZ_p[[G_0^{p^m}]])=
D_{r_0}(G,\bbQ_p).\end{numequation}compatible with the canonical
maps $\frh\hookrightarrow Lie(G)\hookrightarrow D(G,\bbQ_p).$
\end{prop}
\begin{proof}
This reformulation of the Lazard isomorphism is essentially
\cite{ArdakovPHD}, Thm. 5.1.4 (compare also \cite{AW}, Thm. 10.4
and the remarks 10.5). Indeed, according to loc.cit. the map
$\cL_G$ induces an isomorphism between
$\widehat{U(\frh)}_{\bbQ_p}$ and the completion of
$\bbQ_p\Lambda_{G}:=\bbQ_p\otimes_{\bbZ_p}\bbZ_p[[G]]$ with
respect to a certain norm on $\bbQ_p\Lambda_{G}$ constructed in
\cite{DDMS}, chap. 7. By the explicit formulae given in loc.cit.
this norm equals the restriction of $||.||_{\rn}$ to
$\bbQ_p\Lambda_{G}$. But the latter space is norm-dense in
$D_{r_0}(G,\bbQ_p)$ which gives the first claim. The compatibility
with the inclusions $\frh\hookrightarrow Lie(G)\hookrightarrow
D(G,\bbQ_p)$ follows from (\ref{compLie}).

\end{proof}

\begin{cor}\label{cor-lazard}
The inverse $\cL^{-1}_G$ induces an injective group homomorphism
$$\cL^{-1}_G:G\hookrightarrow \widehat{U(\frh)}^\times.$$
\end{cor}
\begin{proof}
%According to \cite{Lazard65}, IV.3.2.5 the image of
%$\widehat{U(\frh)}\subset\widehat{U(\frh)}_{\bbQ_p}$ under the
%isomorphism $\cL_G$ equals the {\it saturation} ${\rm
%Sat}\bbZ_p[G]$ of the valued group ring $\bbZ_p[G]$ (cf. also
%\cite{AW}, Remark 10.5(a)).
The inclusion $\bbZ_p[G]\subset {\rm Sat~}\bbZ_p[G]$ identifies
$G$ with a subgroup of units of the ring ${\rm Sat~}\bbZ_p[G]$.
\end{proof}

\end{para}\begin{para}\label{para-HYP}

We now turn back to the setting of section \ref{sect-BB}. Recall
the $o_L$-Lie algebra $\frg$ of the connected split reductive
group scheme $\bG$. Let $\frg_0$ equal $\frg$ but viewed over
$\bbZ_p$ and recall that $e$ is the ramification index of
$L/\bbQ_p$. We assume that we are given a $L$-uniform pro-$p$
group $G$ that satisfies the following hypothesis relative to
$\bG$:

\[\begin{array}{ccc}
  {\rm (HYP)} &  & L_{G}=\pi_L^{k}\frg_0 \\
\end{array}\]
as Lie algebras over $\bbZ_p$ for a natural number $k\geq 1$. We
additionally fix a number $m\geq 1$. In the following we are
studying the $\bbZ_p$-Lie algebra
\begin{numequation}\label{deform}
\frh:=\frac{1}{p}L_{G^{p^m}}=\frac{1}{p}p^mL_{G}=p^m\frac{1}{p}L_{G}=\pi_L^{(m-1)e+k}\frg_0.
\end{numequation}
%Recall here that the Lie algebra $L_{G_0^{p^m}}$ is powerful which
%implies
%\[[\frac{1}{p}L_{G_0^{p^m}},\frac{1}{p}L_{G_0^{p^m}}]\subseteq
%\frac{1}{p}L_{G_0^{p^m}}\] and so $\frh$ really is closed under
%the Lie bracket of $Lie(G)$.

Abbreviating $n:=(m-1)e+k \geq 1$, the Lazard isomorphism applied
to $G^{p^m}$ is an isomorphism
\begin{numequation}\label{ultimate}
\cL_{G_0^{p^m}}:
\widehat{U(\pi_L^n\frg_0)}_K=\widehat{U(\frh)}_K\car
D_{r_0}(G_0^{p^m}). \end{numequation}

Fix an ordered $o_L$-basis $\fry_1,...,\fry_d$ of $\frg$ and a
$\bbZ_p$-basis $v_1,...,v_{[L:\bbQ_p]}$ of $o_L$ with $v_1=1$.
Given $\al\in\bbN_0^{[L:\bbQ_p]d}$, we form the elements
$$\frY^\al:=(v_1\fry_1)^{\al_{11}}\cdot\cdot\cdot(v_{[L:\bbQ_p]}\fry_d)^{\al_{[L:\bbQ_p]d}}$$
inside $U(\frg)$. The $K$-Banach algebra
$\widehat{U(\pi_L^n\frg_0)}_{K}=\widehat{U(\pi_L^n\frg_0)}\otimes_{\bbZ_p}
K$ is given by all series
\[ \lambda:=\sum_{\al\in\bbN_0^{[L:\bbQ_p]d}} d_\al \frY^\al \]
where $d_\al\in K$ with $|d_\al|\cdot
|\pi_L|^{-n\cdot|\al|}\rightarrow 0$ for
$|\alpha|\rightarrow\infty$.
%i.e.  $|d_\al|s^{|\al|}\rightarrow 0$ with $s:=p^{l}$.
%The left-hand side of this equation contains
%$U(\frg_0)\otimes_{\bbQ_p}L=U(\frg_0\otimes_{\bbQ_p}L)$ as
%subalgebra.
It contains the kernel $J$ of the map
\[\frg_0\otimes_{\bbZ_p}o_L\rightarrow\frg,~~\frx\otimes a\mapsto
a\frx.\] If the set $J$ injects into some algebra we denote by
$\langle J\rangle$ the two-sided ideal generated by $J$ in this
algebra. Reasoning as in (the proof of) \cite{SchmidtVECT}, Prop.
2.4 the quotient algebra $\widehat{U(\pi_L^n\frg_0)}_{o_L}/\langle
J\rangle$ is given by the $\pi_L$-adically complete algebra of all
series
\[ \lambda:=\sum_{\beta\in\bbN_0^{d}} e_\beta \frY'^\beta \]
where $e_\beta\in o_L$ with $|e_\beta|\cdot |\pi_L|^{-n\cdot
|\beta|}\rightarrow 0$ for $|\beta|\rightarrow\infty$
%i.e.  $|d_\beta|s^{|\beta|}\rightarrow 0$ with $s:=p^{l}$.
and
$\frY'^{\beta}:=\fry_1^{\beta_{1}}\cdot\cdot\cdot\fry_d^{\beta_{d}}$.
It therefore coincides with
$\widehat{U(\pi_L^n\frg)}=\widehat{U(\frg)_{n}}.$ Now taking into
account the discussion preceding the map (\ref{scalar}), the
Lazard isomorphism (\ref{ultimate}) factores into an isomorphism
\begin{numequation}\label{ultimatequot}
\cL_{G^{p^m}}:
\widehat{U(\frg)_{n,}}_{K}=\widehat{U(\frg)_{n}}\otimes_{o_L}K=\widehat{U(\pi_L^n\frg_0)}_{K}/\langle
J\rangle\car D_{r_0}(G_0^{p^m})/\langle J\rangle=D_{r_0}(G^{p^m})
\end{numequation}
and the inverse $\cL^{-1}_{G^{p^m}}$ yields an injective group
homomorphism $G^{p^m}\hookrightarrow
\widehat{U(\frg)_{n}}^\times.$ For future reference we summarize
this discussion in a proposition.
\begin{prop}\label{prop-enddiscussion}
Let $G$ be a $L$-uniform pro-$p$ group such that
$L_G=\pi_L^{k}\frg_0$ for some $k\geq 1$. Let $m\geq 1$ and put
$n:= (m-1)e+k$. The Lazard isomorphism induces an algebra
isomorphism
$$\cL_{G^{p^m}}: \widehat{U(\frg)_{n,}}_K\car D_{r_0}(G^{p^m})$$
whose inverse yields an injective group homomorphism
$G^{p^m}\hookrightarrow \widehat{U(\frg)_{n}}^\times.$
\end{prop}

\begin{para}
Let us keep the assumptions on the group $G$ appearing in the
preceding proposition. Let $Z(\frg_K)$ be the center of the ring
$U(\frg_K)$ and let $\theta: Z(\frg_K)\ra K$ be an algebra
homomorphism. We suppose that $Z(\frg_K)\hookrightarrow D(G)$ has
image in the center of the ring $D(G)$. For example, this is true
whenever $G$ is an open subgroup in the group of rational points
of a connected algebraic group over $L$ (\cite{ST4}, Prop. 3.7).
We then have the central reduction
$$D(G)_\theta:=D(G)\otimes_{Z(\frg_K),\theta} K$$ and similarly
for any $D_r(G)$. We remark that, since the ring $Z(\frg_K)$ is
noetherian, the ideal of $D(G)$ (and $D_r(G)$) generated by
$\ker\theta$ is finitely generated and hence closed (\cite{ST5},
remark before Prop. 3.6). It follows from loc.cit., Prop. 3.6 that
\begin{numequation}\label{equ-centred}D(G)_\theta\car \varprojlim_r
D_r(G)_\theta.\end{numequation}

\vskip8pt

On the other hand, given a weight
$\lambda\in\Hom_{o_L}(\pi_L^n\frt,o_K)$ we have the central
reduction $\hUlmK$ of $\hUmK$ from section \ref{sect-DC}. We
suppose that $\theta$ is related to $\lambda_K\in\frt^*_K$ via the
(untwisted) Harish-Chandra homomorphism
$$\phi_K: Z(\frg_K)\car S(\frt_K)^{W,\bullet}$$ (\cite{Dixmier},
Thm. 7.4.5). Here, $(W,\bullet)$ refers to the usual dot-action
$$w\bullet \mu:=w(\mu+\rho)-\rho$$ of the Weyl group $W$ on a
weight $\mu\in \frt_K$. According to (\cite{DemazureGabriel},
II.\S6.1.5) and since $char(L)=0$, we have
$U(\frg)^\bG\otimes_{o_L} L =Z(\frg_L)$ which means that $\phi_K$
comes by base change via $o_L\rightarrow K$ from the homomorphism
$\phi: U(\frg)^\bG\rightarrow S(\frt)$ of section \ref{sect-DC}.
This justifies our notation.
\begin{cor}\label{cor-enddiscussion2}
The isomorphism $\cL_{G^{p^m}}$ of the preceding proposition
factores into an isomorphism
$$\hUlmK\car D_{r_0}(G^{p^m})_\theta.$$
\end{cor}
\begin{proof}
Let $R:=(U(\frg)^\bG)_n$. By \cite{AW}, Lemma 6.5 we have the
canonical isomorphism
$$\hUlmK = \widehat{U(\frg)_n}\otimes_{\hat{R}_K,\lambda} K \hskip25pt {\rm (*)} $$ where the
completions are taken, as always, with respect to $p$-adic
topologies. We claim that the induced topology on
$R=(U(\frg)^\bG)_n \subset U(\frg)_n$ is again the $p$-adic
topology. Indeed, since $\bG$ acts $o_L$-linearly on $U(\frg)$ via
its adjoint action, this is obvious for $n=0$. The general case
$n>0$ is reduced to this case when we can show that the natural
map
$$(U(\frg)^{\bG})_n\hookrightarrow (U(\frg)_n)^{\bG}$$ is onto.
Let $x\in U(\frg)_n$ be $\bG$-invariant. Since the deformed
$PBW$-filtration on $U(\frg)_n$ is still a split filtration and
since $\bG$ necessarily acts by strictly filtered isomorphisms we
may assume that $x\in\pi_L^{in}F_iU(\frg)$ for some $i\geq 0$. But
then
$$x\in\pi^{in}_L (F_iU(\frg))^{\bG}=\pi^{in}_L
F_i(U(\frg)^{\bG})\in (U(\frg)^{\bG})_n.$$

By the claim the $p$-adic completion $\hat{R}$ equals the closure
of $R$ inside $\widehat{U(\frg)_n}$. This means that $\hat{R}_K$
equals the closure $\overline{R_K}$ of $R_K=Z(\frg_K)$ inside
$\widehat{U(\frg)_n}_{,K}$. So we deduce from ${\rm (*)}$ and the
preceding proposition the isomorphism
$$\hUlmK\car  D_{r_0}(G^{p^m}) \otimes_{\overline{R_K},\lambda} K =
D_{r_0}(G^{p^m}) \otimes_{R_K,\lambda} K
=D_{r_0}(G^{p^m})_\theta$$ which proves the assertion.
\end{proof}

In the next section we will apply these results in the case where
$G$ equals a certain principal congruence subgroup in the rational
points of the group scheme $\bG$.
\end{para}

% By Prop. 6.3 (e) we have (*): \hat{U}^\lambda_e = \hat{U}^\lambda_e\otimes_{\hat{(U(\frg_{\bbZ_p})^\bG)_e},\lambda} o_K.
% Now the induced topology on D:=(U(\frg_{\bbZ_p})^\bG)_e \subset
%U(\frg_{\bbZ_p})_e is again the p-adic topology: check this for e=0:
% but then we have for all N>0: p^N U(\frg_{\bbZ_p}) \cap U(\frg_{\bbZ_p})^\bG = p^N U(\frg_{\bbZ_p})^\bG since $\bG$ acts \bbZ_p-linear.
% Now for general e>0: see notes.
%Then \hat{D}= closure of (U(\frg_{\bbZ_p})^\bG)_e inside
% \hat{U}^\lambda_e and then, via inverting $\varpi_K$ we have \hat{D}_K = closure of Z(\frg) inside \hat{U}^\lambda_{e,K}.
% All in all the equation (*) becomes, via inverting $\varpi_K$ and observing that \hat{U}^\lambda_{e,K}=D_{r_0}(U_{x_0}):
%\hat{U}^\lambda_{e,K}= D_{r_0}(U_{x_0}) \otimes_{closure of Z(\frg) inside D_{r_0}(U_{x_0}) } K_\lambda = D_{r_0}(U_{x_0})_\theta.

\section{Congruence kernels in semisimple $p$-adic groups}
We assume in this section that $p\neq 2$.

\begin{para}
%For technical reasons we have to assume in this section that
We return to the setting of section \ref{sect-DC}. Consider the
group $\bG(o_L)$ of $o_L$-valued points of the split reductive
group scheme $\bG$. For every $k\geq 1$ we have in $\bG(o_L)$ the
normal congruence subgroup
\[ G(\pi_L^{k}):=\ker\; ( \bG(o_L)\longrightarrow
\bG(o_L/\pi_L^{k}o_L) ).\]
\begin{lemma}
Let $k\geq e$. The locally $L$-analytic group $G(\pi_L^{k})$ is
$L$-uniform. The corresponding $o_L$-lattice in
$Lie(G(\pi_L^{k}))=\frg\otimes_{o_L} L$ is given by
$\pi_L^{k}\frg$.
\end{lemma}
\begin{proof}
The construction in \cite{DDMS}, \S13. Ex. 11 extends from the
split semisimple case to the case of $\bG$. In particular,
$G(\pi_L^{k})$ is an $o_L$-standard group in the sense of loc.cit.
or \cite{B-L}, III.\S7.3 and the corresponding $o_L$-lattice
-mapping onto $G(\pi_L^{k})$ via the exponential map- is given by
$\pi_L^{k}\frg$. Since $k\geq e$, the $\bbZ_p$-Lie algebra
$\pi_L^k\frg_0=p (\pi_L^{k-e}\frg_0)$ is powerful (recall $p\neq
2$) and therefore $G(\pi_L^{k})_0$ is a uniform pro-$p$ group (cf.
\cite{DDMS}, Thm. 9.8).
\end{proof}
According to the lemma we have $L_{G(\pi_L^{k})}=\pi_L^{k}\frg_0$
for all $k\geq e$. Hence the hypothesis {\rm (HYP)} from
(\ref{para-HYP}) is satisfied. Applying Prop.
\ref{prop-enddiscussion} yields the following proposition.
\begin{prop}\label{prop-congruence}
Let $G:=G(\pi_L^{k})$ with $k\geq e$. Let $m\geq 1$ and put $n:=
(m-1)e+k$. There is a canonical algebra isomorphism
$$\cL_{G^{p^m}}: \widehat{U(\frg)_n}_{,K}\car
D_{r_0}(G^{p^m})$$ whose inverse yields an injective group
homomorphism $G^{p^m}\hookrightarrow
\widehat{U(\frg)_{n}}^\times.$
\end{prop}

\end{para}

\vskip8pt

\end{para}
\begin{para}
In the following we let $G=G(\pi_L^k)$ be a congruence subgroup
for $k\geq e$. We have the finite group
$$H_m:=G/G^{p^m}.$$ We define inductively a set $\overline{H}_m$ of representatives for $H_m$ in
$G$ containing $1\in G$ with the property
\begin{numequation}\label{sys-rep}\overline{H}_{m+1}=\{\overline{g}\overline{h}:
\overline{g}\in \overline{H}_{m,m+1}, \overline{h}\in
\overline{H}_m\}\end{numequation} where $\overline{H}_{m,m+1}$ is
a set of representatives for the group $G^{p^m}/G^{p^{m+1}}$. Let
$$s_m: H_m\lra G$$
be the corresponding section of the canonical projection
homomorphism $G\ra H_m$. As usual we write $\overline{h}=s_m(h)$
for $h\in H_m$.

\vskip8pt

As in \ref{para-TAU} we have a map $\tau': H_m\times
H_m\rightarrow G^{p^m}$ by requiring
$$\overline{h}_1\overline{h}_2=\overline{h_1h_2}\cdot\tau'(h_1,h_2)$$ in
$G^{p^m}$ for $\overline{h}_i\in \overline{H}_m$ representatives
of the classes $h_i$. As in the above proposition, we let from now
on
$$n:=n(m):=(m-1)e+k$$
which is $\geq 1$. The group $G$ acts from the right via the
adjoint representation $g\mapsto {\rm Ad}(g^{-1})$ on
$U(\frg)$.\footnote{We use the right action here to match later
with our convention for the action map of a crossed product, cf.
\ref{para-ACT}.} This action descends to the deformation
$U(\frg)_n$ and extends then to the $p$-adic completion
$\widehat{U(\frg)_n}$. Since $\widehat{U(\frg)^\bG}$ is fixed
pointwise by this action we have an action on
$$\hUlm=\widehat{U(\frg)_n}\otimes_{\widehat{U(\frg)^\bG},\lambda}o_K$$
by letting $G$ act on the left factor. We therefore have a group
homomorphism $G\ra {\rm Aut}(\hUlm)$ and, consequently, the map
$$\sigma_u: H_m\stackrel{s_m}{\lra} G\lra {\rm Aut}(\hUlm).$$
On the other hand, there is the composite map
\begin{numequation}\label{equ-tauU}\tau_u: H_m\times
H_m\stackrel{\tau'}{\longrightarrow}
G^{p^m}\stackrel{\cL_{G^{p^m}}^{-1}}{\longrightarrow}
\widehat{U(\frg)_n}^\times\longrightarrow
\hUlm^\times\end{numequation} where the final map is given by
$\frx\mapsto \frx\otimes 1$.
\begin{lemma}\label{lem-rel1}
The action map $\sigma_u$ induces a group homomorphism $$H_m\ra
{\rm Out}(\hUlm).$$ The map $\tau_u$ is a $2$-cocycle with respect
to $\sigma_u$.
\end{lemma}
\begin{proof}
Let $\sigma':G\rightarrow {\rm Aut}(G^{p^m})$ be the right action
given by conjugation of $G$ on the normal subgroup $G^{p^m}$, i.e.
$\sigma'(g): h\mapsto g^{-1}hg$. %%Since the group ring $K[G]$ is a
%crossed product of $K[G^{p^m}]$ by $H_m$ with data $\sigma',\tau'$
%(\cite{Passman}).
It is easy to check that the identities
\[\label{equ-IDENT}\begin{array}{ccc}
  \tau'(xy,z)\cdot\tau'(x,y)^{\sigma'(z)}=\tau'(x,yz)\cdot\tau'(y,z), & & \sigma'(y)\sigma'(z)=\sigma'(yz)\eta(y,z) \\
\end{array}\]
hold for all $x,y,z\in H_m$ where $\eta(y,z)$ denotes the
automorphism induced by right conjugation with $\tau'(y,z)$. It
remains to observe that the homomorphism $\cL^{-1}_{G^{p^m}}$
intertwines the conjugation action on the source with the adjoint
action on the target.
\end{proof}

According to the preceding lemma we have, for each $m\geq 1$, an
associative crossed product
$$A^\lambda_m:=\hUlm*_{\sigma_u,\tau_u} H_m$$ where $n=(m-1)e+k\geq 1$ and $H_m=G/G^{p^m}$.
We emphasize that the datum $\sigma_u,\tau_u$ depends on $m$. We
put $A^\lambda_{m,K}:=A^\lambda_m\otimes_{o_L} K$ so that
$A^\lambda_{m,K}=\hUlmK* H_m.$

\end{para}
\begin{para}\label{para-RESA}
Recall the homomorphism
$$res_u: \widehat{\cU^\lambda_{n+1}}_{,K}\ra\widehat{\cU^\lambda_{n}}_{,K}$$
for all $n$ from the end of section \ref{sect-DC}. Using it we may
define a $K$-linear homomorphism
$$res_a: A^\lambda_{m+1,K}\lra A^\lambda_{m,K}$$ compatible with $res_u$
as follows: by (\ref{sys-rep}) the source is a finite free
$\widehat{\cU^\lambda_{n+1}}_{,K}$-module on basis elements
$\overline{g}\overline{h}\in \overline{H}_{m+1}$ where
$\overline{g}\in \overline{H}_{m,m+1}\subset G^{p^m}$ and
$\overline{h}\in \overline{H}_m$ and we put
$$res_a(\overline{g}\overline{h}):=L_u(\overline{g})\cdot\overline{h}$$
where $L_u$ equals the map
$$G^{p^m}\stackrel{\cL_{G^{p^m}}^{-1}}{\longrightarrow}
\widehat{U(\frg)_n}^\times\longrightarrow \hUlm^\times $$ used in
the definition of $\tau_u$. It will follow from the proposition
below that the $K$-linear map $res_a$ is actually an algebra
homomorphism.

\vskip8pt

Consider the central character $$\theta: Z(\frg_K)\lra K$$ that is
induced from $\lambda_K\in\frt^*_K$ via the Harish-Chandrda map
$\phi_K$ and recall the central reductions $D_{r_m}(G)_\teta$.
Recall (\ref{cor-crossed}) the canonical inclusion
$D_{r_0}(G^{p^m})\hookrightarrow D_{r_m}(G).$ Cor.
\ref{cor-enddiscussion2} and Prop. \ref{prop-congruence} imply the
isomorphism $$\cL_{G^{p^m}}: \hUlmK\car D_{r_0}(G^{p^m})_\theta.$$
The two maps compose to an algebra homomorphism
$$\hUlmK\lra
D_{r_m}(G)_\theta.$$
\begin{prop}\label{prop-ident}
The homomorphism extends to an algebra isomorphism
$$A^\lambda_{m,K}\car D_{r_m}(G)_\teta$$ which is compatible with
variation in $m$. In particular, there is a canonical algebra
isomorphism $$A^\lambda_K:=\varprojlim_m A^\lambda_{m,K}\car
D(G)_\theta.$$
\end{prop}
\begin{proof}
%The first statement is a combination of cor. \ref{cor-crossed} and
%prop. \ref{prop-congruence}. The second follows from this and our
%explanations in section \ref{sect-dis}.
Source and target are crossed products involving the same group
$H_m$. It can be checked that the corresponding actions and
twistings are intertwined by the isomorphism $\cL_{G^{p^m}}$. This
yields the first assertion. The second assertion follows by
inspection. The final assertion follows from (\ref{equ-centred})
together with the fact that $r_m\uparrow 1$ for $m\ra\infty$.
\end{proof}

\vskip8pt

\end{para}

\begin{para}
Recall that the sheaf $\tilde{\cD}$ is (left) $\bG$-equivariant
and, a fortiori, left $G$-equivariant. Since $G$ acts
$o_L$-linearly this structure descends to the deformation
$\tilde{\cD}_n$ and then to the $p$-adic completion
$\widehat{\tilde{\cD}_n}$. Since the image of
$$\widehat{\psi_n}: \hUtm\lra \widehat{\tiDm}$$ lies in the
$G$-invariants of $\widehat{\tilde{\cD}_n}$ (\cite{Milicic93},
\S3) we obtain an equivariant structure on the central reduction
$$\hDml=\widehat{\tiDm}\otimes_{\hUtm,\lambda} o_K$$
by letting $G$ act on the left factor. Let $U\in\cS$. The algebra
$\hDml(U)$ equals the $p$-adic completion of $\Dml(U)$ where
$\Dml=\tiDm\otimes_{U(\frt)_n,\lambda} o_K$. Since $\cS$ is a base
for the topology on $X$, the sheaf $\hDml$ is therefore supported
only on the special fibre $X_{s}=X \times_{o_L}\kappa$ of the
$o_L$-scheme $X$.
%\vskip8pt
%Remark: According to \cite{PrasadYu}, 7.4 and \cite{YuModel}, 2.8
%there is for each integer $k\geq 1$ a closed smooth $o_L$-subgroup
%scheme $\bG_k\subseteq \bG$ which gives the inclusion
%$\bG(p^k)\subseteq\bG(o_L)$ on rational points and whose
%$\bbZ_p$-Lie algebra is given by $p^k\frg$. According to
%\cite{WaterhouseWeisfeiler}, Prop. 1.2 the induced action of each
%of these group schemes stabilizes $X_s$ set-theoretically.
However, the set $X_{s}$ is fixed pointwise by the group
$$G(\pi_L)=\ker\; ( \bG(o_L)\longrightarrow \bG(\kappa)).$$
Since $G\subseteq G(\pi_L)$ the equivariant $G$-structure is
actually an action of $G$ on the sheaf $\hDml$. Precomposing with
the involution $g\mapsto g^{-1}$ we have a homomorphism
$$ \tilde{\sigma}: G\lra {\rm Aut}(\hDml)$$ (we remind the reader another time of our convention regarding the
multiplication in automorphism groups, cf. section
\ref{sect-crossed}). Consequently, we have a map
$$\sigma_\cD: H_m\stackrel{s_m}{\lra} G\lra {\rm
Aut}(\hDml)$$ where $s_m$ is our fixed section of the projection
map $G\ra H_m$. Recall that we have an algebra homomorphism
$$\widehat{\varphi^\lambda_n}:\hUlm\lra \Gamma(X,\hDml)$$
which intertwines the (right) adjoint action on the source with
our right action on the target (cf. \cite{Milicic93}, \S3). We
have a map
$$\tau_\cD: H_m\times
H_m\stackrel{\tau_u}{\longrightarrow}\hUlm^\times
\stackrel{\widehat{\varphi^\lambda_n}}{\longrightarrow}
\Gamma(X,\hDml)^\times.$$
\begin{lemma}\label{lem-OUT}
The action map $\sigma_\cD$ induces a group homomorphism $$H_m\ra
{\rm Out}(\hDml).$$ The map $\tau_\cD$ is a $2$-cocycle with
respect to $\sigma_\cD$.
\end{lemma}
\begin{proof}
Let $\sigma:=\sigma_\cD,\tau:=\tau_\cD$. We know that $\tau_u$ is
a $2$-cocycle with respect to $\sigma_u$. By the intertwining
property of $\widehat{\varphi^\lambda_n}$ mentioned above we
obtain the first of the following two identities:
\[\begin{array}{ccc}
  \tau(xy,z)\cdot\tau(x,y)^{\sigma(z)}=\tau(x,yz)\cdot\tau(y,z), & & \sigma(y)\sigma(z)=\sigma(yz)\eta(y,z), \\
\end{array}\]
$x,y,z\in H_m$ where $\eta(y,z)$ denotes the automorphism of the
sheaf $\hDml$ induced by right conjugation with the unit
$\tau(y,z)$. Let us show the second identity. By definition of
$\tau'$ we have in $G$
$$\overline{y}\cdot \overline{z}=\overline{yz}\cdot\tau'(y,z).$$ Applying the
homomorphism $\tilde{\sigma}$ we obtain
$$\sigma(y)\sigma(z)=\sigma(yz)\tilde{\sigma}(\tau'(y,z)).$$
%where $\tilde{\sigma}(\tau'(y,z))$ denotes the automorphism
%induced by the group element $\tau'(y,z)\in G^{p^m}$ in our right
%action of $\hDml$.
We are therefore reduced to show
$$\eta(y,z)=\tilde{\sigma}(\tau'(y,z))\hskip25pt {\rm (*)}$$ as automorphisms of the
sheaf $\hDml$. Since $\hDml$ is supported on $X_s$, it suffices to
prove this for the restriction of $\hDml$ to $X_s$
(\cite{Iversen}, Prop. 6.4) which we denote by the same symbol.
Choose an open subset $V\in\cS$ and put $U:=V\cap X_s$. It
suffices to check that both sides of ${\rm (*)}$ induce the same
automorphism of $\hDml(U)$. This assertion is then seen most
conceptually by introducing a certain completed skew group ring as
follows (cf. \cite{PSS}, \S3). To ease notation and to make the
comparison with loc.cit. more transparent we will for the rest of
this proof only consider left actions. Since $U$ is contained in
the special fibre, the induced action of $G^{p^m}$ on $X$ fixes
$U$ pointwise whence $G^{p^m}$ acts on $\cO_X(U)$ by ring
automorphisms. The derived action $\frg\rightarrow\cT_X$ of $\bG$
on $X$ induces, by restriction, an action of $\pi_L^n\frg$ on
$\cO_X(U)$ by $o_L$-derivations. We may therefore form the skew
group ring and the skew enveloping algebra
\[\begin{array}{ccc}
  \cO_X(U)\# G^{p^m}, &  & \cO_X(U)\# U(\frg)_n. \\
\end{array}\]

with respect to these actions (\cite{MCR}, 1.5.4, 1.7.10). Both
rings have natural $G^{p^m}$-actions: given $g\in G^{p^m}$ let
$$g.(f\otimes h):=(g.f)\otimes (ghg^{-1})$$
for $f\in\cO_X(U), h\in G^{p^m}$ and
$$g.(f\otimes\frx):=(g.f)\otimes({\rm Ad}(g)(\frx))$$
for $f\in\cO_X(U), \frx\in U(\frg)_n$. We let $$\iota:
U(\frg)_n\hookrightarrow \cO_X(U)\# U(\frg)_n,~\frx\mapsto
1\otimes\frx.$$

\vskip8pt

Let $\widehat{\cO_X(U)}$ be the $p$-adic completion of $\cO_X(U)$.
Denote the base change to $K$ of the $p$-adic completion of the
skew enveloping algebra by $\widehat{\cO_X(U)}_K\#
\widehat{U(\frg)_n}_{,K}$. The $G^{p^m}$-action extends to this
completion. There is a natural $G^{p^m}$-equivariant surjective
algebra homomorphism
$$ pr: \widehat{\cO_X(U)}_K\# \widehat{U(\frg)_n}_{,K}\lra
\widehat{\tilde{\cD}_n}_{,K}(U)$$ compatible with the maps $\iota$
and $\widehat{\varphi_n}_{,K}$ (cf. (\ref{equ-Dtrivial}) and
\cite{Milicic93Preprint}, Lemma C.1.3).

The $G^{p^m}$-action on the Banach space $\widehat{\cO_X(U)}_K$ is
induced from the rational $\bG$-action on $X$. It is therefore
locally analytic as can be seen along the lines of (the proof of)
\cite{ST02c}, Prop. 2.1'. According to \cite{PSS}, 3.2 there is a
natural structure of topological $K$-algebra on the completed
tensor product
$$ \widehat{\cO_X(U)}_K\# D_{r_0}(G^{p^m}):=\widehat{\cO_X(U)}_K\hat{\otimes}_K D_{r_0}(G^{p^m})$$
compatible with the two skew multiplication rings above via the
maps
\[\begin{array}{ccc}
  G^{p^m}\hookrightarrow D_{r_0}(G^{p^m})^\times, &  & \cL_{G^{p^m}}: \widehat{U(\frg)_n}_{,K}\car D_{r_0}(G^{p^m}). \\
\end{array}\]
Note for the second compatibility that $\cL_{G^{p^m}}$ is
compatible with the inclusions $U(\frg)\subset Lie(G^{p^m})\subset
D(G^{p^m})$ (prop. \ref{prop-Lazard}). The inclusion
$$\cO_X(U)_K\#G^{p^m}\hookrightarrow \widehat{\cO_X(U)}_K\#D_{r_0}(G^{p^m})$$ has dense image
(\cite{ST4}, Lemma 3.1) and the $G^{p^m}$-action extends by
continuity to the target. We summarize this discussion in the
following commutative diagram
\[
\xymatrix{ \widehat{\tilde{\cD}_n}_{,K}(U)\ar[d]^{=} &
\widehat{U(\frg)_n}_{,K}\ar[d]^{\iota}\ar[l]_{\widehat{\varphi_n}_{,K}}
\ar[r]_{\simeq}^{\cL_{G^{p^m}}}&
 D_{r_0}(G^{p^m})\ar[d]^{\iota'}
 \\
\widehat{\tilde{\cD}_n}_{,K}(U)&
\widehat{\cO_X(U)}_K\#\widehat{U(\frg)_n}_{,K}\ar[l]^-{pr}\ar[r]^{\simeq}
& \widehat{\cO_X(U)}_K\#D_{r_0}(G^{p^m})}
\]
where $\iota'(\delta):=1\hat{\otimes}\delta.$ The lower horizontal
arrows are $G^{p^m}$-equivariant. Consider an element
$g=\cL_{G^{p^m}}(x)\in G^{p^m}$ and $\lambda\in
\widehat{\cO_X(U)}_K\#\widehat{U(\frg)_n}_{,K}$. We claim
$$g.\lambda =\iota(x)\cdot\lambda\cdot \iota(x)^{-1}.$$ According to the right-hand part of the above
commutative diagram this can be checked in
$\widehat{\cO_X(U)}_K\#D_{r_0}(G^{p^m})$ where the identity reads
$$g. \lambda = \iota'(g)\cdot \lambda\cdot \iota'(g)^{-1}.$$ By a density argument using again \cite{ST4}, Lemma 3.1 it suffices to check
this identity on elements of the form $\lambda= f\otimes\delta_h$
with $f\in\cO_X(U),\delta_h\in G^{p^m}$, i.e. in the skew group
ring $\cO_X(U)\#G^{p^m}$. Denoting by $\cdot$ the skew
multiplication in the latter ring we have
$$g. (f\cdot\delta_h)=(g.f)\cdot(g\delta_h g^{-1})=g\cdot
 f\cdot g^{-1}\cdot (g\delta_h g^{-1})=g\cdot
 (f\cdot\delta_h)\cdot g^{-1}$$
 which is the assertion. Using the left-hand part of the above
commutative diagram we deduce that a group element
$g=\cL_{G^{p^m}}(x)\in G^{p^m}$ acts on
$\widehat{\tilde{\cD}_n}_{,K}(U)$ by conjugation with
$\widehat{\varphi_n}_{,K}(x)$. Passing to central reductions via
$\lambda$ we see that such an element acts on
$\widehat{\cD^\lambda_n}_{,K}(U)$ by conjugation with
$\widehat{\varphi^\lambda_n}_{,K}(x)$. Applying this to the
element $\tau_u(y,z)\in\hUlmK$ and observing that
$\cL_{G^{p^m}}\circ\tau_u=\tau'$ (cf. (\ref{equ-tauU})) we see
that $\tau'(y,z)$ acts from the left on
$\widehat{\tilde{\cD}_n}_{,K}(U)$ by left conjugation with the
unit $\widehat{\varphi^\lambda_n}_{,K}(\tau_u(y,z))=\tau(y,z)$
which is what we want.
\end{proof}

\vskip8pt

To ease notation we abbreviate $\sigma:=\sigma_\cD,\tau:=\tau_\cD$
in the following. By section \ref{sect-crossed} and the lemma we
have, for each $m\geq 1$, a crossed product sheaf of associative
$K$-algebras

$$\Bslm:=\hDmlK*_{\sigma,\tau}H_m$$ where $n=(m-1)e+k$ and $H_m=G/G^{p^m}$. We emphasize that the datum
$\sigma,\tau$ depends on $m$. Since $\hDmlK$ is a sheaf of
coherent rings, so is $\Bslm$.

\end{para}

\begin{para}

Let us assume that the weight $\lambda+\rho\in \frt^*_K$ is
dominant and regular. Since $\Bslm$ is a finite free (left) module
over $\hDmlK$, the equivalence of categories for $\hDmlK$ given in
Thm. \ref{thm-BB} admits the following straightforward version
over the sheaf of rings $\Bslm$. Recall the functor ${\rm
Loc}_\lambda$ of (4.11). Given a coherent $\Gamma(X,\Bslm)$-module
$M$ we may form the $\Bslm$-module
$${\rm L}_\lambda(M):=\Bslm\otimes_{\Gamma(\Bslm)} M.$$ Any element $\mu$ in this
module can be written as $$\sum_{k,i} \lambda_k(\mu)\overline{h}_k
\otimes m_i =\sum_{k,i} \lambda_k(\mu)\otimes \overline{h}_k.m_i$$
with $\lambda(\mu)_k \in \hUlmK, m_i\in M$. The natural
homomorphism
 $${\rm Loc}_\lambda(M)\twoheadrightarrow {\rm L}_\lambda(M)$$
 of $\hDmlK$-modules induced by the embedding $\hDmlK\hookrightarrow\Bslm$ is therefore surjective.
Since ${\rm Loc}_\lambda(M)$ is a coherent $\hDmlK$-module (Thm.
\ref{thm-BB}), ${\rm L}_\lambda(M)$ is a coherent $\Bslm$-module.
In a similar manner we deduce that $\Gamma(\cM)$ is a coherent
$\Gamma(\Bslm)$-module for any coherent $\Bslm$-module $\cM$. We
therefore have the adjoint pair $({\rm L}_\lambda, \Gamma(\cdot))$
between the abelian categories ${\rm coh}(\Gamma(\Bslm))$ and
${\rm coh}(\Bslm)$.
\begin{prop}\label{prop-BBneu} Suppose the weight $\lambda+\rho\in\frt^*_K$ is dominant
and regular. The pair $({\rm L}_\lambda, \Gamma(\cdot))$ induces
mutually inverse equivalence of categories
$${\rm coh}(\Gamma(\Bslm))\car{\rm coh}(\Bslm).$$
\end{prop}
\begin{proof}
By Thm. \ref{thm-BB} the functor $\Gamma(\cdot)$ is exact and has
no kernel. By \cite{BeilinsonGinzburg}, Lemma 2.4 it suffices to
check that the counit $\eta_M: M\rightarrow \Gamma({\rm
L}_\lambda(M))$ is an isomorphism for $M\in {\rm
coh}(\Gamma(\Bslm))$. Since $M$ is coherent, it has a finite free
presentation as $\Gamma(\Bslm)$-module. Since
$\eta_{\Gamma(\Bslm)}$ is obviously an isomorphism and $\Gamma(
{\rm L}_\lambda(\cdot))$ is right-exact, the claim follows from
the Five Lemma.
\end{proof}

\end{para}

\begin{para}

Let $U\in\cS$ and $h\in H_m$. Since $\sigma(h)$ preserves the unit
ball $\hDml(U)$ of the $K$-Banach algebra $\hDmlK(U)$ it acts on
the latter by isometries. Recall (cf. (\ref{equ-resD})) the
algebra homomorphism
$$res_{\cD}: \widehat{\cD^\lambda_{n+1}}_{,K}\lra\hDmlK$$ induced by
$\tilde{\cD}_{n+1}\subseteq\tilde{\cD}_n$. Using this homomorphism
we may define a $K$-linear homomorphism
$$res_{\cA}:\cA^\lambda_{m+1,K}\lra\Bslm$$
compatible with $res_{\cD}$ by a procedure very similar to the one
we used for the rings $A^\lambda_{n,K}$. Indeed, the source is a
finite free $\widehat{\cD^\lambda_{n+1}}_{,K}$-module on basis
elements $\overline{g}\overline{h}\in \overline{H}_{m+1}$ where
$\overline{g}\in \overline{H}_{m,m+1}\subset G^{p^m}$ and
$\overline{h}\in \overline{H}_m$ and we put
$$res_{\cA}(\overline{g}\overline{h}):=L_{\cD}(\overline{g})\cdot\overline{h}$$
where $L_{\cD}$ equals the composite of the map
$$L_u: G^{p^m}\stackrel{\cL_{G^{p^m}}^{-1}}{\longrightarrow}
\widehat{U(\frg)_n}^\times\longrightarrow \hUlm^\times $$ used in
the definition of $\tau_u$ and the map
$\widehat{\varphi^\lambda_n}: \hUlm^\times\ra
\Gamma(X,\hDml)^\times$.

\begin{lemma} The morphism $res_{\cA}$ is multiplicative.
\end{lemma}
\begin{proof}
Let $\overline{g}\in \overline{H}_{m,m+1}\subset G^{p^m},
\overline{h}\in\overline{H}_{m}$ and $\lambda$ a local section of
$\widehat{\cD^\lambda_{n+1}}_{,K}.$ By definition
$$
res_{\cA}(\overline{h}\cdot\sigma(h)(\lambda))=res_{\cA}(\lambda\cdot\overline{h})=res_\cD(\lambda)\cdot\overline{h}$$
and the right-hand side equals, by equivariance of $res_\cD$,
$$\overline{h}\cdot\sigma(h)(res_\cD(\lambda))=res_{\cA}(\overline{h})\cdot res_\cA(\sigma(h)(\lambda)).$$
Similarly,
$$res_{\cA}(\overline{g}\cdot\sigma(g)(\lambda))=res_{\cA}(\lambda\cdot\overline{g})=res_\cD(\lambda)\cdot L_\cD(\overline{g}).$$
We have shown in the proof of lemma \ref{lem-OUT} that the group
element $g$ operates on the local section $res_\cD(\lambda)$ with
respect to the $G$-structure on the sheaf $\hDmlK$ through
conjugation by the global unit
$L_\cD(\overline{g})\in\Gamma(X,\hDmlK)^\times$. This implies
$$res_\cD(\lambda)\cdot L_\cD(\overline{g})=L_\cD(\overline{g})\cdot\sigma(g)(res_\cD(\lambda))= res_{\cA}(\overline{g})\cdot
res_\cA(\sigma(g)(\lambda)).$$ Finally, we have to show that
$res_\cA$ respects products of the form
$\overline{h}_1\cdot\overline{h}_2$ with $h_1,h_2\in
\overline{H}_{m+1}$. Since the map $res_a$ from \ref{para-RESA} is
multiplicative we have $res_a(\overline{h}_1)\cdot
res_a(\overline{h}_2)=res_a(\overline{h_1}\cdot\overline{h_2})$.
Applying the multiplicative map $\widehat{\varphi^\lambda_n}_{,K}$
and the preceding lemma yields the claim.
\end{proof}

We therefore have the projective system $(\Bslm)_m$ of sheaves of
$K$-algebras. By Cor. \ref{cor-conditions1} and Prop.
\ref{prop-hypothesis} it satisfies the conditions (i),(ii),(iii)
of section \ref{sect-crossed}. Hence, letting
$$\Bsl:=\varprojlim_m \Bslm$$ be the projective limit we have the
chain of full abelian subcategories
$${\rm coh}(\Bsl)\subset\sC_{\Bsl}\subset {\rm Mod}(\Bsl)$$ and
the equivalence of categories
\begin{numequation}\label{equ-equivalence2}\Gamma: {\rm coh}((\Bslm)_m)\car\sC_{\Bsl}.\end{numequation}
\end{para}
\begin{para}
We shall need a simple lemma on the behaviour of the categories
${\rm coh}(\Bslm)$ relative to the restriction maps $res_\cA$. Let
$(\cM_m)_m$ be a projective system with $\cM_m\in {\rm
coh}(\Bslm)$. We have a natural map of $\Gamma(\Bslm)$-modules
$$\Gamma(\Bslm)\otimes_{\Gamma(\Bslmp)} \Gamma(\cM_{m+1})\lra \Gamma(\Bslm\otimes_{\Bslmp}
\cM_{m+1}).$$
\begin{lemma}\label{lem-BBneu2}
Suppose the weight $\lambda+\rho\in\frt^*_K$ is dominant and
regular. The above map is an isomorphism.
\end{lemma}
\begin{proof}
The sheaf $\Bslm$ is a sheaf of coherent rings. The $\Bslm$-module
$\Bslm\otimes_{\Bslmp} \cM_{m+1}$ is hence of finite presentation
and therefore coherent (cf. (2.4)). According to the above
proposition, $\Gamma(\cdot)$ is exact on coherent modules over
either $\Bslm$ or $\Bslmp$ and therefore both sides of our map are
right-exact functors in $\cM_{m+1}$. By the above proposition, we
have a global finite presentation $(\Bslmp)^{b}\rightarrow
(\Bslmp)^{a}\rightarrow \cM_{m+1}\rightarrow 0$. The Five Lemma
reduces us therefore to the case $\cM_{m+1}=\Bslmp$ which is
obvious.
\end{proof}
\end{para}

\begin{para}
We now compute the global sections of the sheaf $\Bsl$ under the
hypothesis (H1)-(H3) of (4.10). To this end, recall the projective
system
$$A^\lambda_{m,K}=\hUlmK*_{\sigma_u,\tau_u} H_m$$ where $n=(m-1)e+k$ and
$H_m=G/G^{p^m}$. The algebra homomorphism
$$\widehat{\varphi^\lambda_n}_{,K}:\hUlmK\lra\hDmlK$$ extends to
a linear map
$$A^\lambda_{m,K}\lra\Bslm\hskip25pt {\rm (*)}$$ preserving the basis
elements $\overline{h}$ for $h\in H_m$ on both sides.

\begin{lemma}
The linear map {$\rm (*)$} is multiplicative and compatible with
variation in $m$.
\end{lemma}
\begin{proof}
The map {$\rm (*)$} respects the action maps $\sigma_u,\sigma$
since $\widehat{\varphi^\lambda_n}_{,K}$ intertwines the adjoint
action with the $G$-action on the target. Since moreover
$\tau=\widehat{\varphi^\lambda_n}\circ\tau_u$ the map ${\rm (*)}$
is multiplicative. The final assertion is obvious from
$L_\cD=\widehat{\varphi^\lambda_n}\circ L_u$.
\end{proof}
\begin{cor}\label{cor-H}\label{cor-H2}
If {\rm (H1)-(H3)} hold, then ${\rm (*)}$ induces a algebra
isomorphisms
$$A^\lambda_{m,K}\car\Gamma(X,\Bslm)$$
for all $m$ and, in the projective limit, an algebra isomorphism
$$D(G)_\theta\car\Gamma(X,\Bsl).$$
\end{cor}
\begin{proof}
The first statement follows from Thm. \ref{thm-globalsections}.
The second statement follows then from the preceding lemma and
Prop. \ref{prop-ident}.
\end{proof}

\end{para}

\begin{para}
Let now $G=G(\pi_L^k)$ be an {\it any} congruence subgroup for
$k\geq 1$. The above discussion for the case $k\geq e$ extends in
an obvious way to the general case as follows. Let
$\tG=G(\pi_L^{\tilde{k}})$ for $\tilde{k}\geq\max (e,k)$ and let
$H_m, A_{m,K}^\lambda, A^\lambda_K, \cA^\lambda_{m,K},
\cA^\lambda_K$ be formed relative to $\tG$ as described above. We
have the finite group $G_m=G/\tG^{p^m}$ containing $H_m$ as a
normal subgroup. We choose a system of representatives $\cR$ for
$G_m/H_m$ in $G$ containing $1\in G$ and let
$\overline{G}_m=\overline{H}_m\cdot\cR$ (in obvious notation) be
the induced system of representatives of $G_m$ in $G$. This
extends the maps $s_m$ and $\tau'$ from $H_m$ to $G_m$. We
therefore obtain extensions of $\sigma_u$ and $\tau_u$ to $G_m$
and corresponding crossed products
$B^\lambda_{m,K}=\widehat{\cU_n^\lambda}_{,K}*G_m$ and
$\cB^\lambda_{m,K}=\widehat{\cD^\lambda_n}_{,K}*G_m.$ According to
\cite{Passman}, Lemma 1.3 these are actually crossed products over
$A^\lambda_{m,K}$ and $\cA^\lambda_{m,K}$ relative to the finite
group $G_m/H_m=G/\tG$ respectively. Passing to the limit over $m$
we obtain crossed products $$B^\lambda_{K}=A^\lambda_K * (G/\tG)
\hskip10pt {\rm and}\hskip10pt \cB^\lambda_K=\cA^\lambda_K *
(G/\tG).$$ Using $D(G)_\theta=D(\tG)_\theta *(G/\tG)$ one deduces,
in case $(H1)-(H3)$ hold, without difficulty an algebra
isomorphism $D(G)_\theta\car \Gamma(X,\cB^\lambda_K)$ extending
Cor. \ref{cor-H2}. Finally, if $\lambda+\rho\in\frt^*_K$ is
dominant and regular, we have the analogues of Prop.
\ref{prop-BBneu} and Lem. \ref{lem-BBneu2} for the projective
system $(\cB^\lambda_{m,K})_m$ which follow with the same proofs.

\end{para}

\section{The equivalence of categories}
Let in the following
$$G=G(\pi_L^{k})=\ker\; ( \bG(o_L)\longrightarrow
\bG(o_L/\pi_L^{k}o_L) )$$ be a congruence subgroup for some $k\geq
1$. We assume $p\neq 2$ and that the hypothesis (H1)-(H3) of
(4.10) are satisfied.
\begin{para}
Recall that $Z(\frg_K)$ equals the center of the universal
enveloping algebra $U(\frg_K)$ and that $Z(\frg_K)$ lies in the
center of the ring $D(G)$
%We recall some notions related to
%the classical {\it Harish-Chandra isomorphism}. To begin with let
%$S(\frt_K)$ be the symmetric algebra of $\frt_K$ and let
%$S(\frt_K)^W$ be the subalgebra of Weyl invariants. The classical
%Harish-Chandra map is an algebra isomorphism
%$$Z(\frg_K)\car S(\frt_K)^W$$ relating central characters and highest weights of irreducible
%highest weight $\frg$-modules in a meaningful way (\cite{Dixmier},
%7.4). Given a linear form $\chi\in\frt^*$ we let
%$$\sigma(\chi): Z(\frg)\rightarrow L$$ denote the central character
%associated with $\chi$ via this Harish-Chandra map.
(\cite{ST4}, Prop. 3.7). We fix a character
$$\theta: Z(\frg_K)\longrightarrow K$$ and the central reduction
\[ D(G)_\theta:=D(G)\otimes_{Z(\frg_K),\theta} K\]
of $D(G)$. As we have explained the algebra
$D(G)_\theta=\varprojlim_m D_{r_m}(G)_\theta$ is again
Fr\'echet-Stein and its coadmissible modules
$$\sC_{G,\theta}=\sC_G\cap {\rm Mod}(D(G)_\theta)$$ are in duality
with admissible locally analytic $G$-representations having
infinitesimal character $\theta$.

\vskip8pt

Remark: Building on a $p$-adic version due to Ardakov-Wadsley
(\cite{AW}, \S8) of the famous Quillen's lemma it is shown in the
preprint \cite{DospSchraen} that any topologically irreducible
admissible locally analytic $G$-representation admits, up to a
finite extension of $K$, an infinitesimal character.

\vskip8pt

Up to a finite extension of $K$ there is an element
$\lambda\in\frt_K^*$ that pulls back to $\theta$ under the
Harish-Chandra homomorphism $\phi_K$. We fix such a $\lambda$.
There is then a minimal $m(\lambda)\geq 0$ such that
$\lambda(\pi_L^{m(\lambda)}\frt)\subseteq o_K$, i.e.
$$\lambda\in\Hom_{o_L}(\pi_L^{m}\frt,o_K)$$ for all $m>
m(\lambda)$. We tacitly restrict in the following to numbers $m>
m(\lambda)$. According to the final paragraph in the preceding
section, we may associate to $G$ a sheaf of rings $\cA^\lambda_K$
on $X$ which comes with an isomorphism
$$D(G)_\theta\car \Gamma(X,\cA^\lambda_K)$$
and an abelian category of coadmissible sheaves
$\sC_{\cA^\lambda_K}$. It allows to introduce the functor
\[\begin{array}{ccc}
{\rm L}_\lambda: {\rm Mod}(D(G)_\theta)\rightarrow
{\rm Mod}(\Bsl), &  & M\mapsto \Bsl\otimes_{D(G)_\theta }M. \\
\end{array}\] It is left
adjoint to the global section functor $\Gamma(X,.)$ and therefore
right exact. The two units of the adjunction induce natural maps
\begin{equation}\label{equ-adjunction2}
\begin{array}{ccc}
 M\lra \Gamma\circ {\rm L}_\lambda (M),  &  &  {\rm L}_\lambda\circ\Gamma(\cM)\lra\cM \\
\end{array}
\end{equation}
for any $M\in {\rm Mod}(D(G)_\theta)$ and $\cM\in {\rm
Mod}(\Bsl)$. We do not know whether ${\rm
L}_\lambda(M)\in\sC_{\Bsl}$ for a general $M\in\sC_{G,\theta}$.
However, we still have following.
\begin{thm}\label{thm-eigen}
Suppose the weight $\lambda+\rho\in\frt^*_K$ is dominant and
regular. The functor $\Gamma(X,.)$ induces an equivalence of
categories
\[\sC_{\Bsl}\car \sC_{G,\theta}.\] A quasi-inverse is given by the functor $${\rm \tilde{L}}_\lambda: M\mapsto \varprojlim_m~
(\Bslm\otimes_{D(G)_\theta} M).$$ If $M$ is a finitely presented
$D(G)_\theta$-module the natural morphism ${\rm L}_\lambda(M) \car
{\rm \tilde{L}}_\lambda(M)$ is an isomorphism.
\end{thm}
\begin{proof}
Let $(\cM_m)_m$ be a projective system with coherent
$\cA^\lambda_{m,K}$-modules $\cM_m$. Put $M_m:=\Gamma(\cM_m)\in
{\rm Mod}^{\rm fg}(A^\lambda_{m,K})$. According to the Cor.
\ref{lem-BBneu2} the natural map
$$\Blm\otimes_{\Blmp}M_{m+1}\lra M_m$$ is an isomorphism precisely
when the natural map $$\Bslm\otimes_{\Bslmp}\cM_{m+1}\lra \cM_m$$
is an isomorphism, i.e. precisely when $(\cM_m)_m\in {\rm
coh}((\Bslm)_m)$. This means that the collection of category
equivalences ${\rm coh}(\Bslm)\car {\rm Mod}^{\rm fg}(\Blm)$
induced by $\Gamma(\cdot)$ induces a category equivalence
$$(\cM_m)_m\mapsto (\Gamma(\cM_m))_m$$

between ${\rm coh}((\Bslm)_m)$ and the abelian category of
families $(M_m)_m$ of finitely generated $\Blm$-modules with the
property $\Blm\otimes_{\Blmp}M_{m+1}\car M_m$. By definition the
latter equals the category of coherent sheafs for the distribution
algebra $D(G)_\theta$. We therefore have the equivalence
$$\sC_{\Bsl}\car\sC_{G,\theta}$$ given by the composite
$$\cM\mapsto (\Bslm\otimes_{\Bsl}\cM)_m\mapsto (\Gamma(\Bslm\otimes_{\Bsl}\cM))_m\mapsto \varprojlim_m
\Gamma(\Bslm\otimes_{\Bsl}\cM)$$ and the right-hand term equals
$\Gamma(\varprojlim_m\Bslm\otimes_{\Bsl}\cM)=\Gamma(\cM)$
according to Prop. \ref{prop-co}. From the definitions we read off
that a quasi-inverse is given by the functor $${\rm
\tilde{L}}_\lambda: M\mapsto \varprojlim_m~
(\Bslm\otimes_{D(G)_\theta} M).$$ We have a natural morphism $f:
{\rm L}_\lambda(M)\ra {\rm \tilde{L}}_\lambda(M)$ induced by the
maps $\Bsl\ra\Bslm$ where the target is coadmissible. If $M$ is
finitely presented, the source is also coadmissible. Since the
morphism becomes an isomorphism for all $m$ after applying
$\Bslm\otimes_{\Bsl}(\cdot)$, it is an isomorphism.
\end{proof}

\begin{cor} Let $Rep(G)_\theta$ be the category
    of admissible locally analytic $G$-representations over $K$ with
    infinitesimal character $\theta$. Let $\lambda\in\frt^*_K$ be a weight that corresponds to $\theta$ under the (untwisted)
    Harish-Chandra homomorphism. Suppose $\lambda+\rho$ is dominant and regular. There is an equivalence of
    categories $$\sC_{\Bsl} \car Rep(G)_\theta$$
    given by $ \cM\mapsto \Gamma(X,\cM)'_b$.
\end{cor}
\end{para}
\begin{para}
The completion $\hat{U}(\frg_K)$ of $U(\frg_K)$ with respect to
all submultiplicative seminorms is called the {\it Arens-Michael
envelope} of the $p$-adic Lie algebra $\frg_K$ and was studied in
the papers \cite{SchmidtBGG},\cite{SchmidtSTAB}. It is a
Fr\'echet-Stein algebra and we have the abelian category
$\sC(\hat{U}(\frg_K)_\theta)$ of coadmissible modules over the
central reduction $\hat{U}(\frg_K)_\theta$ at the infinitesimal
character $\theta$. Arguing as in \cite{SchmidtBGG}, Prop. 3.2.3
and using the formula ${\rm (*)}$ in the proof of \cite{ST5},
Prop. 3.7 we see that $\hat{U}(\frg_K)_\theta$ is canonically
isomorphic to the projective limit of the
$\widehat{\cU^\lambda_{n}}_{,K}$. On the other hand, we may form
the projective limit $\widehat{\cD^\lambda_K}$ of the $\hDmlK$.
According to Cor. \ref{cor-conditions1} we have the chain of full
abelian subcategories
$${\rm coh}(\widehat{\cD^\lambda_K})\subset\sC_{\widehat{\cD^\lambda_K}}\subset {\rm Mod}(\widehat{\cD^\lambda_K}).$$
The commutative diagram at the end of section $4$ induces a map
$\hat{U}(\frg_K)_\theta\ra\widehat{\cD^\lambda_K}$. According to
Thm. \ref{thm-BB}, it is an isomorphism on global sections in case
(H1)-(H3) hold. Completely similar to the above discussion one may
prove in this case that, if $\lambda+\rho$ is dominant and
regular, then the functor
$M\mapsto\varprojlim_n~(\widehat{\cU^\lambda_{n}}_{,K}\otimes_{\hat{U}(\frg_K)_\theta}
M)$ induces an equivalence of categories
$$ \sC(\hat{U}(\frg_K)_\theta)\car
\sC_{\widehat{\cD^\lambda_K}}$$  with quasi-inverse given by
$\Gamma(X,.)$. This equivalence is compatible with Thm.
\ref{thm-eigen} in the obvious way.

\end{para}

\section{Dimension computations}

As an application of the above methods we prove in this final
section a locally analytic version of Smith's theorem
(\cite{Smith}). First, we have to recall some notation and
properties concerning the canonical dimension. The canonical
dimension (or rather the codimension) in the context of locally
analytic representations appears first in \cite{ST5},\S8.

\begin{para}
We recall (cf. \cite{LVO}, chap. III.) the notion of an Auslander
regular ring. Let $R$ be an arbitrary associative unital ring. For
any (left or right) $R$-module $N$ the {\it grade} $j_R(N)$ is
defined to be either the smallest integer $l$ such that
Ext$_R^l(N,R)\neq 0$ or $\infty$. A left and right noetherian ring
$R$ is called left and right {\it Auslander regular} if its left
and right global dimension is finite and if every finitely
generated left or right $R$-module $N$ satisfies {\it Auslander's
condition}: for any $l\geq 0$ and any $R$-submodule
$L\subseteq$\,Ext$_R^l(N,R)$ one has $j_R(L)\geq l$.

\vskip8pt

In the following the term module always means {\it left} module.
Noetherian rings are two-sided noetherian and other ring-theoretic
properties such as Auslander regular are used similarly.

\vskip8pt

Let $R$ be an Auslander regular ring and $M$ a finitely generated
$R$-module. The number
\[ d_R(M):=gld(R)-j_R(M)\]
is called the {\it canonical dimension} of $M$. The map $M \mapsto
d_R(M)$ is a finitely partitive exact dimension function in the
sense of \cite{MCR}, \S6.8.4 and \S8.3.17. For this and more
details on the canonical dimension for finitely generated modules
over an Auslander regular ring (or more generally, an
Auslander-Gorenstein ring) we refer to \cite{Levasseur}.

\vskip8pt

A commutative noetherian ring is Auslander regular if and only if
it is regular (\cite{LVO}, III.2.4.3). Let $R$ be an associative
ring endowed with a separated exhaustive $\bbZ$-filtration by
additive subgroups such that $R$ is complete in the filtration
topology. If the associated graded ring of $R$ is Auslander
regular of (left and right) global dimension $d$ then $R$ is
Auslander regular and has (left and right) global dimension $\leq
d$ ([loc.cit.], II.2.2.1, II.3.1.4, III.2.2.5).

\end{para}

\begin{para}

Let now $G$ be a locally $L$-analytic group which is $L$-uniform.
Recall that the rings underlying the $K$-Banach algebras
$D_{\rn}(G^{p^m})$ and $D_{r_m}(G)$ are noetherian (\cite{ST5},
Thm. 4.5).
\begin{prop}\label{prop-ARI}
Let $m\geq 1$ and let $A$ denote one of the rings
$D_{\rn}(G^{p^m})$ or $D_{r_m}(G)$. Then $A$ is Auslander regular
of (left and right) global dimension equal to $d={\rm dim}_L G$. A
finitely generated $A$-module $M$ is finite dimensional over $K$
if and only if $d_{A}(M)=0$.
\end{prop}
\begin{proof}
By Lem. \ref{lem-free} the ring $D_{r_m}(G)$ is a finite free
extension of $D_{\rn}(G^{p^m})$ on a basis a system of
representatives for the group $G/G^{p^m}$. By \cite{ST5}, Lem. 8.8
and \cite{SchmidtAUS}, Cor. 7.3 it therefore suffices to prove the
proposition in the case $A=D_{\rn}(G^{p^m})$. By Lem.
\ref{lem-graded} and the discussion above the ring $A$ is seen to
be Auslander regular of global dimension $\leq d$ and, at the same
time, a complete doubly filtered $K$-algebra with
$Gr(A)=k[u_1,...,u_d]$. Furthermore, functoriality of $D(\cdot)$
gives a $K$-algebra homomorphism
\[ D(G^{p^m})\rightarrow D(\{1\})=K\]
which extends to $A$ and hence, $A$ admits a nonzero module of
finite $K$-dimension. Now \cite{AW}, Prop. 9.1 completes the
proof.
\end{proof}
For future reference we repeat the following ingredient of the
preceding proof.
\begin{cor}\label{cor-grade}
If $M$ denotes a finitely generated $D_{r_m}(G)$-module $M$ then
\[d_{D_{r_m}(G)}(M)=d_{D_{\rn}(G^{p^m})}(M).\]
\end{cor}

\end{para}

\begin{para}\label{valuesofr}

As before we now let $\bG$ be a connected split reductive group
scheme over $o_L$ with Lie algebra $\frg$. We assume that
(H1)-(H3) from (4.10) hold. Let $\bG'_{\bbC}$ denote the complex
derived algebraic group of $\bG$ and let $\frg'_{\bbC}$ be its Lie
algebra. It is known that there is a unique non-zero coadjoint
$\bG'_{\bbC}$-orbit in $(\frg'_{\bbC})^*$ of minimal dimension
%It is called the {\it
%minimal nilpotent orbit}
(\cite{CollingwoodMcGovern}, Rem. 4.3.4) and its dimension is an
even integer $\geq 2$. Let $r$ denote half this dimension.
According to work of A. Joseph (\cite{Joseph}) the values of $r$
are well-known. The following table is taken from loc.cit. with
the two corrections noted in \cite{JosephMINORBIT}.
\[\begin{tabular}{c|ccccccccc}
$\Phi$ &$A_l$& $B_l$ & $C_l$ & $D_l$ & $E_6$ & $E_7$ & $E_8$ & $F_4$ & $G_2$\\
\hline

${\rm dim~} \frg'_{\bbC}$ & $l^2+2l$ &$2l^2+l$ &$2l^2+l$ &$2l^2-l $ &$78$&$133$&$248$&$52$&$14$\\
$r$ & $l$ & $2l-2 $ & $l$ & $2l-3$ & $11$ & $17$ & $29$ & $8$ &
$3.$
\\

\end{tabular}\]

\end{para}

\begin{para}

Let $\tilde{G}$ be from now on a locally $L$-analytic group whose
$L$-Lie algebra $Lie(\tG)$ is isomorphic to
$\frg_L:=L\otimes_{o_L}\frg$. Let us identify these Lie algebras
via such an isomorphism. Let $d={\rm dim}_L Lie(\tG)$. Choose
$k\geq 1$ sufficiently large such that $\exp_{\tilde{G}}$ is
defined on $\Lambda:=\pi_L^{k}\frg$ and such that
\[G:=\exp_{\tilde{G}}(\Lambda)\]
is an open subgroup of $\tilde{G}$ which is $L$-uniform. Then
\[L_{G}=\Lambda=\pi_L^{k}\frg_0\]
as Lie algebras over $\bbZ_p$ where $\frg_0$ equals $\frg$ but
viewed over $\bbZ_p$. In other words, the condition {\rm (HYP)}
from \ref{para-HYP} is satisfied. Let $m\geq 1$. According to
Prop. \ref{prop-enddiscussion} we have a canonical algebra
isomorphism
$$\cL_{G^{p^m}}: \widehat{U(\frg)_n}_{,K}\car
D_{r_0}(G^{p^m},K)$$ where $n:= (m-1)e+k \geq 1.$ At this point we
recall a deep theorem of \cite{AW}.
\begin{thm} {\rm (Ardakov/Wadsley)} \label{thm-AW}
Let $A:=\widehat{U(\frg)_n}_{,K}$. If $M$ is a finitely generated
(left) $A$-module with $d_A(M)\geq 1$, then $d_A(M)\geq r$.
\end{thm}
\begin{proof}
It is straightforward to check that the results of loc.cit.,
9.3/4/5/6 extend from the semisimple case considered there to the
case of $\bG$. We may therefore assume, by passing to a finite
field extension of $K$, that $M$ is a (finitely generated)
$\widehat{\cU^\lambda_n}_{,K}$-module for a weight
$\lambda\in\frt^*_{K}$ such that $\lambda+\rho$ is dominant. There
is a canonical decomposition as Lie algebras $\frg=\fra \times
\frg'$ with abelian $\fra$ and a semisimple algebra $\frg'\neq 0$.
The (completed deformed) central reduction
$\widehat{\cU^\lambda_n}$ depends only on $\frg'$. The claim
follows now as in the proof of loc.cit., Thm. 9.10.
\end{proof}
Remark: Thm. 9.10 of loc.cit. is the analogue of Smith's theorem
for $p$-adically completed universal enveloping algebras. It
allows to prove a similar analogue for certain Iwasawa algebras
culminating in the main result of loc.cit.

\vskip8pt

By Cor. \ref{cor-grade} we have
\begin{cor}\label{cor-main}
If $M$ is a finitely generated (left) $D_{r_m}(G)$-module with
$d_{D_{r_m}(G)}(M)\geq 1$ then $d_{D_{r_m}(G)}(M)\geq r$.
\end{cor}
As we have pointed out in section \ref{sect-dis} the projective
system of noetherian $K$-Banach algebras $D_{r_m}(G)$ for $m\geq
1$ defines the Fr\'echet-Stein structure of $D(G)$. By Prop.
\ref{prop-ARI} the rings $D_{r_m}(G)$ are Auslander regular with
global dimension equal to $d$. This implies (\cite{ST5}, \S8) that
the category of coadmissible modules $\sC_{\tG}$ is equipped with
a well-behaved {\it canonical codimension function} defined as
${\rm codim}(M):=j_{D(G)}(M)$ for a coadmissible module $M$. This
number does not depend on the choice of $G$. Let ${\rm
dim}(M):=d-{\rm codim}~{M}$. Hence $0\leq {\rm dim}(M)\leq d$ if
$M\neq 0$ and ${\rm dim}(0)=-\infty$.

\vskip8pt

\begin{thm}\label{thm-smith}
Let $M\in\sC_{\tG}$ with ${\rm dim}(M)\geq 1$. Then ${\rm
dim}(M)\geq r$.
\end{thm}
\begin{proof}
According to \cite{ST5}, Lem. 8.4
%in fact this is explicitly stated on p. 46 of ST5
we have ${\rm dim} (M):=\sup_{m\geq 1} d_{D_{r_m}(G)}(M_m)$ and
hence ${\rm dim} (M)=d_{D_{r_{m'}}(G)}(M_{m'})\geq 1$ for a
particular $m'\geq 1$. Now apply Cor. \ref{cor-main} to $M_{m'}.$
\end{proof}

\begin{prop}
Let $M\in\sC_{\tG}$. Then $M$ is zero-dimensional if and only if
the coherent sheaf of $M$ consists of finite-dimensional
$K$-vector spaces.
\end{prop}
\begin{proof}
If $M$ is zero-dimensional then $d_{D_{r_m}(G)}(M_m)=0$ for all
$m\geq 1$ and the claim follows from Prop. \ref{prop-ARI}. The
other implication is \cite{SchmidtAUS}, Thm. 8.4.
\end{proof}

\end{para}

\bibliographystyle{plain}
\bibliography{mybib}

\end{document}